\documentclass[11pt,reqno]{amsart}
\usepackage{amssymb}
\usepackage[all]{xy}
\usepackage{hyperref}
\usepackage{enumitem}
\usepackage{xcolor}
\newtheorem{thm}{Theorem}[section]
\newtheorem{lem}[thm]{Lemma}
\newtheorem{prop}[thm]{Proposition}
\newtheorem{cor}[thm]{Corollary}

\theoremstyle{definition}
\newtheorem{dfn}[thm]{Definition}
\newtheorem{ex}[thm]{Example}

\theoremstyle{remark}
\newtheorem{remark}[thm]{Remark}

\newcommand{\CD}{{\mathcal{D}}}

\newcommand{\CF}{{\mathcal{F}}}
\newcommand{\CH}{{\mathcal{H}}}
\newcommand{\CS}{{\mathcal{S}}}
\newcommand{\CI}{{\mathcal{I}}}

\newcommand{\CL}{{\mathcal{L}}}

\newcommand{\CT}{{\mathcal{T}}}

\newcommand{\CB}{{\mathcal{B}}}

\newcommand{\CR}{{\mathcal{R}}}

\newcommand{\af}{\alpha}
\newcommand{\bt}{\beta}
\newcommand{\gm}{\gamma}

\newcommand{\ld}{\lambda}

\newcommand{\prim}{\operatorname{Prim}}

\newcommand{\sing}{\operatorname{sing}}

\newcommand{\N}{{\mathbb{N}}}
\newcommand{\T}{{\mathbb{T}}}

\begin{document}


\title{Condition (K) for Boolean dynamical systems}

\author[T. M. Carlsen]{Toke Meier Carlsen}
\address{Department of Sciences and Technology \\
University of Faroe Islands, Vestara Bryggja 15, FO-100 T\'orshavn
\\Faroe Islands} \email{toke.carlsen\-@\-gmail.\-com }

\author[E. J. Kang]{Eun Ji Kang$^{\dagger}$}
\thanks{Research partially supported by NRF-2017R1D1A1B03030540$^{\dagger}$}
\address{
Research Institute of Mathematics, Seoul National University, Seoul 08826, 
Korea} \email{kkang33\-@\-snu.\-ac.\-kr }

\keywords{$C^*$-algebras of Boolean dynamical systems, graph $C^*$-algebras, Condition (K), gauge-invariant ideals, real rank zero, maximal tail, ultrafilter cycle, primitive ideal space, topological dimension zero, ideal property}

\subjclass[2010]{46L05, 46L55}

\begin{abstract}
We generalize Condition (K) from directed graphs to Boolean dynamical systems and show that a locally finite Boolean dynamical system $(\CB,\CL,\theta)$ with countable $\CB$ and $\CL$ satisfies Condition (K) if and only if every ideal of its $C^*$-algebra is gauge-invariant, if and only if its $C^*$-algebra has the (weak) ideal property, and if and only if its $C^*$-algebra has topological dimension zero. As a corollary we prove that if the $C^*$-algebra of a locally finite Boolean dynamical system with $\CB$ and $\CL$ countable, either has real rank zero or is purely infinite, then $(\CB, \CL, \theta)$ satisfies Condition (K). We also generalize the notion of maximal tails from directed graph to Boolean dynamical systems and use this to give a complete description of the primitive ideal space of the $C^*$-algebra of a locally finite Boolean dynamical system that satisfies Condition (K) and has countable $\CB$ and $\CL$.
\end{abstract}

\maketitle

\setcounter{equation}{0}

\section{Introduction}
\subsection{Background}
$C^*$-algebras associated to Boolean dynamical systems were introduced in \cite{COP} as a generalization of graph $C^*$-algebras. The class of $C^*$-algebras of Boolean dynamical systems also contains the class of ultragraph $C^*$-algebras, the class of $C^*$-crossed products of Cantor minimal systems, $C^*$-algebras of second countable zero-dimensional topological graphs, $C^*$-algebras of shift spaces, as well as many labeled graph $C^*$-algebras. 

One of the merits of $C^*$-algebras of Boolean dynamical systems is that many of the results about graph $C^*$-algebras can be generalized to $C^*$-algebras of Boolean dynamical systems. In \cite{COP}, the K-theory of the $C^*$-algebra of a Boolean dynamical system was computed, the set of gauge-invariant ideals of the $C^*$-algebra of a Boolean dynamical system was determined, Condition (L) for a Boolean dynamical system was introduced as a generalization of Condition (L) for directed graphs, and it was shown that a Boolean dynamical system $(\CB,\CL,\theta)$ with countable $\CB$ and $\CL$ satisfies Condition (L) if and only if its $C^*$-algebra satisfies the Cuntz--Krieger uniqueness theorem.

Condition (K) for directed graphs was introduced in \cite{KPRR}. A directed graph satisfies Condition (K) if and only if every ideal of its $C^*$-algebra is gauge-invariant \cite[Corollary 3.8]{BHRS}, and if and only if its $C^*$-algebra has real rank zero \cite[Theorem 3.5]{Jeong}. If a $C^*$-algebra has real rank zero, then it has the ideal property (\cite[Remark 2.1]{CPR}) and is $K_0$-liftable (\cite[Definition 3.1]{CPR}). It is proven in \cite[Proposition 2.11 and Theorem 4.2]{CPR} that the converse holds for a separable purely infinite $C^*$-algebra. The weak ideal property was introduced in \cite{PP} where it was also shown that the ideal property implies the weak ideal property. It was proven in \cite[Theorem 2.8]{PP2} that a $C^*$-algebra with the weak ideal property has topological dimension zero.

In \cite{BPRS}, the notion of a maximal tail of a row-finite directed graph with no sinks was introduced, and it was shown that if a row-finite directed graph $E$ with no sinks satisfied Condition (K), then there is a bijective correspondence between the primitive ideals of the $C^*$-algebra of $E$ and the maximal tails of $E$. In addition, a topology on the set of maximal tails was introduced, and it was proved that the previously mentioned bijective correspondence becomes a homeomorphism if  the set of maximal tails is equipped with this topology and the set of primitive ideals of the $C^*$-algebra of $E$ is equipped with the hull-kernel topology. This was generalized to arbitrary directed graphs in \cite{DT}.

\subsection{The contents of this paper}
In this paper we introduce Condition (K) for Boolean dynamical systems (Definition~\ref{def:cond(K)}) and prove that a locally finite Boolean dynamical system $(\CB,\CL,\theta)$ with countable $\CB$ and $\CL$ satisfies Condition (K) if and only if the quotient Boolean dynamical system $(\CB/\CH, \CL, \theta)$ satisfies Condition (L)  for every hereditary saturated ideal $\CH$ of $\CB$, and if and only if every ideal of its $C^*$-algebra $C^*(\CB,\CL,\theta)$ is gauge-invariant (Theorem~\ref{equivalent:(K)}).
 
We also generalize the notion of maximal tails from directed graph to Boolean dynamical systems (Definition~\ref{def:maximal tail}), and show that if a locally finite Boolean dynamical system $(\CB,\CL,\theta)$ with countable $\CB$ and $\CL$ satisfies Condition (K), then there is a bijective correspondence between the primitive ideals of the $C^*$-algebra of $(\CB,\CL,\theta)$ and the maximal tails of $(\CB,\CL,\theta)$ (Proposition~\ref{max-bij-prim}); and we introduce a topology and the set of maximal tails of $(\CB,\CL,\theta)$ (Proposition~\ref{max-top}) such that this correspondence becomes a homeomorphism when the set of maximal tails of $(\CB,\CL,\theta)$ is equipped with this topology and the set of primitive ideals of the $C^*$-algebra of $(\CB,\CL,\theta)$ is equipped with the hull-kernel topology (Theorem~\ref{max-homeo-prim}).

Using these results, we also prove that 
a locally finite Boolean dynamical system $(\CB,\CL,\theta)$ with countable $\CB$ and $\CL$ satisfies Condition (K) 
 if and only if its $C^*$-algebra has the (weak) ideal property, if and only if its $C^*$-algebra has topological dimension zero, and if and only if its $C^*$-algebra has no quotient that contains a corner that is isomorphic to $M_n(C(\T))$ for some $n\in\N$ (Theorem~\ref{equivalent:(K):top dim zero}). As a corollary we prove that if the $C^*$-algebra of a locally finite Boolean dynamical system with $\CB$ and $\CL$  countable, either has real rank zero or is purely infinite, then $(\CB, \CL, \theta)$ satisfies Condition (K) (Corollary~\ref{cor}).

\subsection{Further discussions}
There are plenty of Boolean dynamical systems that satisfy Condition (K) without their $C^*$-algebras being purely infinite, for instance Boolean dynamical systems that give rise to AF-algebras (see \cite{JKK}) and Boolean dynamical systems that give rise to Cantor minimal systems (see \cite{JKKP}). In contrast to this, the authors do not know of any Boolean dynamical system that satisfies Condition (K) without its $C^*$-algebra having real rank zero.

The reason that we have to assume that $\CB$ and $\CL$ are countable in many of our results is that this is an assumption in the Cuntz--Krieger Uniqueness Theorem  \cite[Theorem 9.9]{COP}. The authors are currently working on a paper \cite{CaK3} in which they prove that the Cuntz--Krieger Uniqueness Theorem holds without the assumption that $\CB$ and $\CL$ are countable, and generalize most of the main results of this paper to Boolean dynamical systems $(\CB,\CL,\theta)$ where $\CB$ and $\CL$ are not necessarily countable (and to the generalized Boolean dynamical systems introduced in \cite{CaK2}). Even though the authors do not know of any interesting Boolean dynamical systems with $\CB$ or $\CL$ uncountable, this might be interesting for future applications.

The authors are convinced that if a (generalized) Boolean dynamical system satisfies Condition (K), then its $C^*$-algebra is $K_0$-liftable. However, to prove this one faces the problem that a gauge-invariant ideal of the $C^*$-algebra of a Boolean dynamical system is not necessarily Morita equivalent to a $C^*$-algebra of a Boolean dynamical system. The authors have in \cite{CaK2} enlarged the class of $C^*$-algebras of Boolean dynamical systems by introducing the notation of relative generalized Boolean dynamical systems and constructed corresponding $C^*$-algebras to these relative generalized Boolean dynamical systems. The authors are planning to show that any gauge-invariant ideal of a $C^*$-algebra of a (relative generalized) Boolean dynamical system is Morita equivalent to the $C^*$-algebra of a relative generalized Boolean dynamical system, and to use this to prove that the $C^*$-algebra of a (generalized) Boolean dynamical system that satisfies Condition (K) is $K_0$-liftable.

If it is correct that the $C^*$-algebra of a (generalized) Boolean dynamical system that satisfies Condition (K) is $K_0$-liftable, then it would follow that if the $C^*$-algebra of a (generalized) Boolean dynamical system is separable and purely infinite, then it has real rank zero.

\subsection{The organization of this paper} 
The rest of the paper is organized in the following way: 
In section \hyperref[preliminary]{2} we recall some preliminary results about Boolean dynamical systems and their $C^*$-algebras. In section \hyperref[ultrafilter cycles]{3} and \hyperref[maximal tails]{4}, we introduce the notions of ultrafilter cycles and maximal tails, respectively.  
In section \hyperref[Condition (K)]{5} we define Condition (K) for Boolean dynamical systems. In section \hyperref[gau-inv-ideal]{6},
we prove that 
a necessary and  sufficient   condition to Condition (K) of a locally finite Boolean dynamical system $(\CB,\CL, \theta)$ with countable $\CB$ and $\CL$  is that every  ideal of $C^*(\CB, \CL, \theta)$ is gauge-invariant.
In section \hyperref[primitive ideal space]{7} we completely characterize the primitive ideal space of the $C^*$-algebras of Boolean dynamical systems. In section \hyperref[top dim zero]{8} we show that a locally finite Boolean dynamical system $(\CB, \CL, \theta)$ with countable $\CB$ and $\CL$ satisfies Condition (K) if and only if its $C^*$-algebra has the (weak) ideal property, if and only if its $C^*$-algebra has topological dimension zero.
 We also  illustrate some of the introduced concepts with a recurring example throughout the paper. 

\vskip 1pc 

\section{Preliminaries}\label{preliminary} 

For the convenience of the reader, we shall in this section briefly recall the definition of a Boolean dynamical system and the $C^*$-algebra of a Boolean dynamical system as well of some basic results about $C^*$-algebras of a Boolean dynamical systems from \cite{COP}.

We let $\N$ denote the set of positive integers.
\subsection{Boolean algebras}

A {\em Boolean algebra} \cite[Definition 2.1]{COP} is a set $\CB$ with a distinguished element $\emptyset$ and maps $\cap: \CB \times \CB \rightarrow \CB$, $\cup: \CB \times \CB \rightarrow \CB$ and $\setminus: \CB \times \CB \rightarrow \CB$ such that $(\CB,\cap,\cup)$ is a distributive lattice, $A\cap\emptyset=\emptyset$ for all $A\in\CB$, and $(A\cap B)\cup (A\setminus B)=A$ and $(A\cap B)\cap (A\setminus B)=\emptyset$ for all $A,B\in\CB$.
The Boolean algebra $\CB$ is called {\em unital} if there exists $1 \in \CB$ such that $1 \cup A = 1$ and $1 \cap A=A$ for all $A \in \CB$ (often, Boolean algebras are assumed to be unital and what we here call a Boolean algebra is often called a \emph{generalized Boolean algebra}).

We call $A\cup B$ the \emph{union} of $A$ and $B$, $A\cap B$ the \emph{intersection} of $A$ and $B$, and $A\setminus B$ the \emph{relative complement} of $B$ with respect to $A$.
A subset $\CB' \subseteq \CB$ is called a {\em Boolean subalgebra} if $\emptyset\in\CB'$ and $\CB'$ is closed under taking union, intersection and the relative complement. A Boolean subalgebra of a Boolean algebra is itself a Boolean algebra.

We define a partial order on $\CB$ as follows: for $A,B \in \CB$,
$$A \subseteq B ~~~\text{if and only if}~~~A \cap B =A. $$   
Then $(\CB, \subseteq)$ is a partially ordered set, and $A\cup B$ and $A\cap B$ are the least upper-bound and the greatest lower-bound of $A$ and $B$ with respect to the partial order $\subseteq$. If a family $\{A_{\ld}\}_{\ld \in \Lambda}$ of elements from $\CB$ has a least upper-bound, then we denote it by $\cup_{\ld \in \Lambda} A_\ld$. If $A\subseteq B$, then we say that $A$ is a \emph{subset of} $B$.
 
A non-empty subset $\CI$ of $\CB$ is called  an {\em ideal} \cite[Definition 2.4]{COP} if the following two conditions holds.
\begin{enumerate}
\item[(i)] If $A, B \in \CI$, then $A \cup B \in \CI$.
\item[(ii)] If $A \in \CI$ and $ B \in \CB$, then   $A \cap B \in \CI$. 
\end{enumerate}
An ideal $\CI$ of a Boolean algebra $\CB$ is a Boolean subalgebra. For $A \in \CB$, the ideal generated by $A$ is defined by $\CI_A:=\{ B \in \CB : B \subseteq A\}.$

If $\CI$ is an ideal of a Boolean algebra $\CB$, then the relation
\begin{eqnarray}\label{equivalent relation} A \sim B \iff A \cup A' = B \cup B'~\text{for some}~ A', B' \in \CI
\end{eqnarray}
defines an equivalence relation on $\CB$ (see \cite[Definition 2.5]{COP}). We denote by $[A]$ the equivalence class of $A \in \CB$ and by $\CB / \CI$ the set of all equivalent classes of $\CB$. It is easy to check that $\CB / \CI$ becomes a Boolean algebra with operations defined by $[A]\cap [B]=[A\cap B]$, $[A]\cup [B]=[A\cup B]$, and $[A]\setminus [B]=[A\setminus B]$.
The partial order $\subseteq$ on $\CB / \CI$ is characterized by  
 \begin{align*} [A] \subseteq [B] & \iff A \subseteq B\cup W  ~\text{for some}~  W \in \CI \\
                                    &\iff  [A] \cap [B] =[A].
  \end{align*}

A {\em filter} \cite[Definition 2.6]{COP} $\xi$ in a Boolean algebra $\CB$ is a non-empty subset $\xi \subseteq \CB$  such that 
\begin{enumerate}
\item[$\mathbf{F0}$] $\emptyset \notin \xi$,
\item[$\mathbf{F1}$] if $ A \in \xi$ and  $A \subseteq B$, then $B \in \xi$,
\item[$\mathbf{F2}$] if $A,B \in \xi$, then $A \cap B \in \xi$.
\end{enumerate}
If in addition $\xi$ satisfies 
\begin{enumerate}
\item[$\mathbf{F3}$] if $A \in \xi$  and $B,B' \in \CB$ with $A=B \cup B'$, then either $B \in \xi$ or $B' \in \xi$,
\end{enumerate}
then it is called an {\em ultrafilter} \cite[Definition 2.6]{COP} of $\CB$. A filter is an ultrafilter if and only if it is a maximal element in the set of filters with respect to inclusion. We write $\widehat{\CB}$ for the set of all ultrafilters of $\CB$. Notice that if $A\in\CB\setminus\{\emptyset\}$, then $\{B\in\CB:A\subseteq B\}$ is a filter, and it then follows from Zorn's Lemma that there is an ultrafilter $\eta\in\widehat{\CB}$ that contains $A$. For $A\in\CB$, we let $Z(A):=\{\xi\in\widehat{\CB}:A\in\xi\}$ and we equip $\widehat{\CB}$ with the topology generated by $\{Z(A): A\in\CB\}$. Then $\widehat{\CB}$ is a totally disconnected locally compact Hausdorff space, $\{Z(A): A\in\CB\}$ is a basis for the topology, and each $Z(A)$ is compact and open. 

We now give a simple example of a Boolean algebra and illustrate some of the concepts introduced above. We shall return to this example throughout the paper.

\begin{ex}\label{ex1}
Let $\mathcal{A}:=\{1,2,3\}$ and let $X:=\mathcal{A}^\N$ be the full one-sided shift space on $\mathcal{A}$. Equip $X$ with the product topology. Then $X$ is a Cantor set (i.e. it is a second countable compact Hausdorff space with no isolated points and with a basis of compact open sets). Let $\CB$ be the set of compact open subsets of $X$. Then $\CB$ is a Boolean algebra with unit $X$ where $\cup$, $\cap$ and $\setminus$ are the usual set operations, and $\emptyset$ is the empty set. 

The map $U\mapsto \CI_U:=\{A\in\CB:A\subseteq U\}$ is a bijection between the set of open subsets of $X$ and the set of ideals of $\CB$. If $U\in\CB$, then the map $A\mapsto [A]$ is a Boolean isomorphism from $\CI_{X\setminus U}$ to the quotient $\CB/\CI_U$.

The map $C\mapsto \CF_C:=\{A\in\CB:C\subseteq A\}$ is a bijection between the set of non-empty compact subsets of $X$ and the set of filters in $\CB$, and $x\mapsto \hat{x}:=\{A\in\CB:x\in A\}$ is a bijection from $X$ to $\widehat{\CB}$. If $\Lambda$ is a subset of $\CB$, then the family $\{A\}_{A\in\Lambda}$ has a least upper-bound if and only if $\cup_{A\in\Lambda}A\in\CB$ (i.e., if and only if $\cup_{A\in\Lambda}A$ is compact), in which case $\cup_{A\in\Lambda}A$ is the least upper-bound of $\{A\}_{A\in\Lambda}$.
\end{ex}

\subsection{Boolean dynamical systems}

A map $\phi: \CB \rightarrow \CB'$ between two Boolean algebras is called a {\em Boolean homomorphism} if $\phi(A \cap B)=\phi(A) \cap \phi(B)$, $\phi(A \cup B)=\phi(A) \cup \phi(B)$, and $\phi(A \setminus B)=\phi(A) \setminus \phi(B)$ for all $A,B \in \CB$.

If $\CB$ is a Boolean algebra, then a Boolean homomorphism $\theta: \CB \rightarrow \CB $ is an {\em action} on $\CB$ if $\theta(\emptyset)=\emptyset$. An action $\theta$ has {\em compact range} \cite[Definition 3.1]{COP} if $\{\theta(A)\}_{A \in \CB}$ has a least upper-bound. We denote by $\CR_{\theta}$ this least upper-bound if it exists. An action $\theta$ has {\em closed domain} \cite[Definition 3.1]{COP} if there exists $\CD_{\theta} \in \CB$ such that $\theta(\CD_{\theta})=\CR_{\theta}$.
Notice that if an action $\theta: \CB \rightarrow \CB $ has compact range and $\CB$ has a unit, then $\theta$ has closed domain.

Given a set $\CL$ and any $n \in \N$, we define $\CL^n:=\{(\af_1, \dots, \af_n): \af_i \in \CL\}$ and $\CL^*:=\cup_{n \geq 0} \CL^n$, where $\CL^0:=\{\emptyset \}$. We define $|\af|$ to be $n$ if $\alpha\in\CL^n$. For $\af=(\af_1, \dots, \af_n)$, $\beta=(\beta_1,\dots,\beta_m) \in \CL^*$, we will usually write $\af_1 \cdots \af_n$ instead of $(\af_1, \dots, \af_n)$ and use $\alpha\beta$ to denote the word $\af_1 \cdots \af_n\beta_1\cdots\beta_m$ (if $\alpha=\emptyset$, then $\alpha\beta:=\beta$; and if $\beta=\emptyset$, then $\alpha\beta:=\alpha$). For $k\in\N$, we let $\alpha^k:=\alpha\alpha\cdots\alpha$ where the concatenation on the right has $k$ terms. Similary we let $\alpha^0:=\emptyset$. For $1\leq i\leq j\leq |\af|$, we denote by $\af_{[i,j]}$ the sub-word $\af_i\cdots \af_j$ of  $\af=\af_1\af_2\cdots\af_{|\af|}$, where $\af_{[i,i]}=\af_i$.

A {\em Boolean dynamical system} \cite[Definition 3.3]{COP} is a triple $(\CB,\CL,\theta)$ where $\CB$ is a Boolean algebra, $\CL$ is a set, and $\{\theta_\af\}_{\af \in \CL}$ is a set of actions on $\CB$ such that for $\af=\af_1 \cdots \af_n \in \CL^*\setminus\{\emptyset\}$, the action $\theta_\af: \CB \rightarrow \CB$ defined as $\theta_\af:=\theta_{\af_n} \circ \cdots \circ \theta_{\af_1}$ has compact range and closed domain. 
Given any $\af \in \CL^*\setminus\{\emptyset\}$, we write 
$\CR_\af:=\CR_{\theta_\af}$. We also define $\theta_\emptyset:=\text{Id}$. 
 
For $B \in \CB$, we define 
$$\Delta_B:=\{\af \in \CL:\theta_\af(B) \neq \emptyset \} ~\text{and}~  \ld_B:=|\Delta_B|.$$
We say that $A \in \CB$ is a {\em regular} \cite[Definition 3.5]{COP} set if for any $\emptyset \neq B \in \CI_A$, we have $0 < \ld_B < \infty$.  If $A \in \CB$ is not regular, then it is called a {\em singular} set. We write $\CB_{reg}$ for the set of all regular sets. Notice that $\emptyset\in\CB_{reg}$.

A Boolean dynamical system $(\CB,\CL, \theta)$ is {\em locally finite} \cite[Definition 3.6]{COP} if for every $ \xi \in \widehat{\CB}$ there exists $A \in \xi$ such that $\ld_A < \infty$. Notice that if $\CL$ is finite, then $(\CB,\CL, \theta)$ is locally finite.
 
\begin{ex}\label{ex2}
Let $\mathcal{A}$, $X$ and $\CB$ be as in Example~\ref{ex1}. If $x=(x_k)_{k\in\N}\in X$ and $a\in\mathcal{A}$, then we denote by $ax$ the element of $X$ with $(ax)_1=a$ and $(ax)_k=x_{k-1}$ for $k>1$. Let $\CL:=\{1,2,i,j\}$. Define maps $\theta_a:\CB\to\CB$ for $a\in\CL$ by 
\begin{align*}
\theta_1(A)&:=\{1x:x\in A,\ x_1\in\{1,2\}\},\\
\theta_2(A)&:=\{2x:x\in A,\ x_1\in\{1,2\}\},\\
\theta_i(A)&:=\{x:x\in A,\ x_1=3\},\\
\theta_j(A)&:=\{1x:x\in A,\ x_1=3\}.
\end{align*}
Then $(\CB,\CL,\theta)$ is a Boolean dynamical system. Since $\CL$ is finite, $(\CB,\CL,\theta)$ is locally finite. Moreover, $\lambda_A\in\{2,4\}$ for any $\emptyset\ne A\in\CB$, so $\CB_{reg}=\CB$.
\end{ex}

\subsection{$C^*$-algebras associated with Boolean dynamical systems}

\label{def-rep}
A {\em Cuntz-Krieger representation} \cite[Definition 3.7]{COP} of a Boolean dynamical system $(\CB,\CL,\theta)$
is a family of projections $\{P_A\,:\, A\in \CB\}$ and
partial isometries
$\{S_\af \,:\, \af \in \CL\}$ in a $C^*$-algebra such that for $A, B\in \CB$ and $\af, \bt \in \CL$,
\begin{enumerate}
\item [(i)] $P_{\emptyset}=0$, $P_{A\cap B}=P_AP_B$, and
$P_{A\cup B}=P_A+P_B-P_{A\cap B}$,
\item[(ii)] $P_A S_\af=S_\af P_{\theta_\af(A)}$,
\item[(iii)] $S_\af^*S_\bt=\delta_{\af,\bt}P_{\CR_\af}$,
\item[(iv)]\label{CK4}  $P_A=\sum_{\af \in \Delta_A} S_\af P_{\theta_\af(A)}S_\af^*$ if $A \in \CB_{reg}$.
\end{enumerate}
A representation is called {\it faithful} if $P_A \neq 0$ for all $A \in \CB$. 

 
It is shown in \cite[Theorem 5.8]{COP} that 
if  $(\CB,\CL,\theta)$ is a Boolean dynamical system, then 
there exists a universal Cuntz-Krieger representation $\{s_\af,p_A\}$ of $(\CB, \CL,\theta)$ (here universal means that if $\{S_\af,P_A\}$ is a Cuntz-Krieger representation, then there is a $*$-homomorphism $\phi$ from the $C^*$-algebra generated by $\{s_\af,p_A\}$ to the the $C^*$-algebra generated by $\{S_\af,P_A\}$ such that $\phi(s_\af)=S_\af$ for $\af\in\CL$ and $\phi(p_A)=P_A$ for $A\in\CB$). We let $C^*(\CB, \CL,\theta)$ denote the $C^*$-algebra generated by a universal representation of $(\CB, \CL,\theta)$ and call it the {\em Cuntz-Krieger Boolean  $C^*$-algebra}, or just the \emph{$C^*$-algebra}, of
the Boolean dynamical system  $(\CB,\CL,\theta)$. 

\begin{remark}\label{basics} 
Let $(\CB,\CL,\theta)$ be a Boolean dynamical system.  
\begin{enumerate}
\item The universal property of  $C^*(\CB,\CL,\theta)=C^*(s_a, p_A)$ 
defines a strongly continuous action 
$\gm:\mathbb{T} \rightarrow \text{Aut}(C^*(\CB,\CL,\theta))$,  
called the  {\it gauge action}, such that
$$\gm_z(s_\af)=zs_\af ~ \text{ and } ~ \gm_z(p_A)=p_A$$
for $\af \in \CL$ and $A\in \CB$. 

\item If $\CB$ and $\CL$ are countable, then $C^*(\CB,\CL,\theta)$ is separable.

\end{enumerate}
\end{remark}

Let $(\CB, \CL,\theta)$ be a Boolean dynamical system and let  $\af=\af_1 \cdots \af_{|\af|} \in \CL^*\setminus\{\emptyset
\}$ and $A \in \CB\setminus\{\emptyset\}$.
 \begin{enumerate}
 \item[(i)] A $(\af,A)$ is called a {\em cycle} \cite[Definition 9.5]{COP} if $B=\theta_{\af}(B)$ for all $B \subseteq A$. 
 \item[(ii)] 
  A cycle $(\af,A)$ has an \emph{exit} if for there is a $t \leq |\af|$ and a $B\in\CB$ such that $\emptyset \neq B \subseteq \theta_{\af_1 \cdots \af_t}(A)$ and $\Delta_B\ne\{\af_{t+1}\}$ (where $\af_{|\af|+1}:=\af_1$).
 \item[(iii)] 
  A cycle $(\af,A)$ has no exits \cite[Definition 9.5]{COP} if and only if for all $t \leq |\af|$ and all $\emptyset \neq B \subseteq \theta_{\af_1 \cdots \af_t}(A)$, we have  $B \in \CB_{reg}$ with $\Delta_B=\{\af_{t+1}\}$ (where $\af_{|\af|+1}:=\af_1$).
 \item[(iv)]  We say that $(\CB, \CL,\theta)$ satisfies {\em Condition (L)} \cite[Definition 9.5]{COP} if every cycle has an exit (in \cite{COP}, this condition is called \emph{Condition $(L_\CB)$}). 
 \end{enumerate}

\begin{ex}\label{ex3}
We continue Example~\ref{ex1} and Example~\ref{ex2}. A pair $(\alpha, A)$ is a cycle if and only if $\alpha=i^n$ for some $n\in\N$ and $A\subseteq \{x\in X: x_1=3\}$. Since $\Delta_A=\{i,j\}$ for any $A\subseteq \{x\in X: x_1=3\}$, every cycle in $\CB$ has an exit, and $(\CB, \CL,\theta)$ satisfies Condition (L).
\end{ex}

We need the following easy strengthening of \cite[Theorem 9.9]{COP}.
 
\begin{thm}\label{CK uniqueness thm}  \cite[Cuntz-Krieger Uniqueness Theorem]{COP} 
Let $(\CB,\CL,\theta)$ be a Boolean dynamical system such that $\CB$ and $\CL$ are countable. If $(\CB,\CL,\theta)$ satisfies Condition (L), then for any $*$-homomorphism $\pi:C^*(\CB,\CL,\theta):=C^*(s_a, p_A) \to B$, the following are equivalent.
\begin{enumerate}
\item[(i)]  $\pi(p_A) \neq 0$ for all $ \emptyset \neq A \in \CB $.
\item[(ii)] $\pi$ is injective.
\end{enumerate}
\end{thm} 

\begin{proof}
(ii)$\implies$ (i) follows from \cite[Corollary 5.3]{COP}. 

(i)$\implies$(ii): It follows from \cite[Theorem 9.9]{COP} that it suffices to show that $\pi(s_\alpha p_A s_\alpha^*) \neq 0$ for all $\alpha\in\CL^*$ and all $\emptyset \neq A \in \CB$ with $A\subseteq\mathcal{R}_\alpha$. Suppose for contradiction that $\pi(s_\alpha p_A s_\alpha^*) = 0$ for some $\alpha\in\CL^*$ and some $\emptyset \neq A \in \CB$ with $A\subseteq\mathcal{R}_\alpha$. Then 
$$\pi(p_A)=\pi(s_\alpha^*s_\alpha p_A s_\alpha^*s_\alpha)=\pi(s_\alpha^*)\pi(s_\alpha p_A s_\alpha^*)\pi(s_\alpha)=0.$$
This shows that (i)$\implies$(ii).
\end{proof}

\subsection{Gauge-invariant ideals of $C^*(\CB,\CL,\theta)$}
We now recall the characterization given in \cite{COP} of the gauge-invariant ideals of the $C^*$-algebra of a locally finite Boolean dynamical system. 

Let $(\CB,\CL,\theta)$ be a Boolean dynamical system. An ideal $\CH$ of $\CB$ is said to be {\em hereditary} if $\theta_{\af}(A) \in \CH$ for $A \in \CH$ and $ \af \in \CL$, and {\em saturated} if $A \in \CH$ whenever $A \in \CB_{reg}$ and $\theta_{\af}(A) \in \CH$ for all $\af \in \Delta_A$.

Given a hereditary ideal  $\CH$  of $\CB$, if we define $\theta_{\af}([A]):=[\theta_{\af}(A)]$ for all $[A] \in \CB/\CH$ and $ \af \in \CL$, then  $(\CB / \CH, \CL,\theta)$ is a Boolean dynamical system (\cite[Proposition 10.7]{COP}). We call it a {\em quotient Boolean dynamical system} of $(\CB,\CL,\theta)$.

\begin{ex}\label{ex4}
We continue Example~\ref{ex1}, Example~\ref{ex2} and Example~\ref{ex3}. Recall that the map $U\mapsto \CI_U:=\{A\in\CB:A\subseteq U\}$ is a bijection between the set of open subsets of $X$ and the set of ideals of $\CB$. If $U$ is a non-empty open subset of $X$, then $\CI_U$ is hereditary if and only if $\{x\in X:x_1\in\{1,2\}\}\subseteq U$ in which case $\CI_U$ is also saturated. 

For an open subset $U$ of $X$ with $\{x\in X:x_1\in\{1,2\}\}\subseteq U$, if we let $\CB_U$ be the set of compact and open subsets of $X\setminus U$ (where the latter is equipped with the subspace topology), then $\CB_U$ equipped with the usual set operations is a Boolean algebra, and the map $A\setminus U\mapsto [A]$ is an isomorphism from $\CB_U$ to $\CB/\CI_U$. 

For an open subset $U$ of $X$ with $\{x\in X:x_1\in\{1,2\}\}\subseteq U$, the action $\theta$ of $\CL$ on $\CB/\CI_U$ is given by $\theta_1([A])=[\theta_1(A)]=\emptyset$, $\theta_2([A])=[\theta_2(A)]=\emptyset$, $\theta_j([A])=[\theta_j(A)]=\emptyset$, and $\theta_i([A])=[\theta_i(A)]=[A]$ for any $A\in\CB$. It follows that if $A\in\CB$ with $A\setminus U\ne\emptyset$, then $(i,[A])$ is a cycle with no exits. So if $U$ is an open subset of $X$ with $\{x\in X:x_1\in\{1,2\}\}\subseteq U$, then $(\CB/\CI_U,\CL,\theta)$ does not satisfy Condition (L).
\end{ex}
 
For a hereditary saturated ideal $\CH$ of $\CB$, we denote by $I_\CH$ the ideal   of $C^*(\CB,\CL,\theta)$ generated by the projections $\{p_A : A \in \CH\}$. Given an ideal $I$ of $C^*(\CB,\CL, \theta)$, we let $\CH_I:=\{A \in \CB: p_A \in I\}$. Then $\CH_I$ is an ideal of $\CB$.

It is shown in \cite[Proposition 10.11]{COP} that if $(\CB,\CL,\theta)$ is a locally finite Boolean dynamical system, then the maps $I \mapsto \CH_I$ and $\CH \mapsto I_{\CH}$ define a one-to-one correspondence between the set of all nonzero gauge-invariant ideals of $C^*(\CB,\CL,\theta)$ and the set of all non-empty hereditary saturated ideals of $\CB$.

\section{Ultrafilter cycles}\label{ultrafilter cycles}

Our definition of Condition (K) for Boolean dynamical systems relies on the notion of \emph{ultrafilter cycles} which we now introduce.

\begin{dfn}
Let $(\CB, \CL, \theta)$ be a Boolean dynamical system. We say that a pair $(\alpha,\eta)$, where $\alpha\in \CL^*\setminus\{\emptyset\}$ and $\eta\in\widehat{\CB}$, is an \emph{ultrafilter cycle} if $\theta_\alpha(A)\in\eta$ for all $A \in \eta$.
\end{dfn}

If $\eta\in\widehat{\CB}$ and $\alpha\in\CL^*$, then $\widehat{\theta}_\alpha(\eta):=\{A\in\CB:\theta_\alpha(A)\in\eta\}$ is either empty or an ultrafilter. In fact, $\widehat{\theta}_\alpha(\eta)\in\widehat{\CB}$ if and only if $\CR_\alpha\in\eta$ or $\alpha=\emptyset$. Let 
$$\widehat{\CR}_\alpha:=\{\eta\in\widehat{\CB}:\CR_\alpha\in\eta\}$$
if $\alpha\ne\emptyset$, and $\widehat{\CR}_\alpha:=\widehat{\CB}$ if $\alpha=\emptyset$. 
Then $\eta\mapsto \widehat{\theta}_\alpha(\eta)$ is a map from $\widehat{\CR}_\alpha$ to $\widehat{\CB}$ which we denote by $\widehat{\theta}_\alpha$. Moreover, $\widehat{\theta}_\alpha(\widehat{\CR}_\alpha)=\{\eta\in\widehat{\CB}: \theta_\alpha(A)\ne\emptyset\text{ for all }A\in\eta\}$.

\begin{lem}\label{ultra filter cycle lemma} 
Let $(\CB,\CL,\theta)$ be a Boolean dynamical system.
\begin{enumerate}
 \item If $(\af,\eta)$ is an ultrafilter cycle, then $(\af^k, \eta)$ is an ultrafilter cycle for all $k \in \N$.  
\item  A pair $(\af, \eta)$ is an ultrafilter cycle if and only if $\eta=\widehat{\theta}_{\af}(\eta)$, and if and only if $\eta=\widehat{\theta}_{\af^k}(\eta)$ for all $k \in \N$.
\item Let $(\af,A)$ be a cycle and $\eta$ be an ultrafilter of $\CB$ such that $A \in \eta$. Then $(\af, \eta)$ is an ultrafilter cycle. 
\end{enumerate}
\end{lem}

\begin{proof}(1): Follows by induction on $k$.

(2): If $\eta=\widehat{\theta}_\af(\eta)$, then obviously $(\af,\eta)$ is an ultrafilter cycle. Conversely, if $(\af,\eta)$ is an ultrafilter cycle, then $\eta \subseteq \widehat{\theta}_{\af}(\eta),$ and hence $\eta = \widehat{\theta}_{\af}(\eta)$ since $\eta$ is an ultrafilter. (2) therefore follows from (1).

(3): Let $B \in \eta$. Then we have that $$ \eta \ni A \cap B =\theta_{\af}(A \cap B)= \theta_{\af}(A) \cap \theta_{\af}(B) \subseteq \theta_{\af}(B), $$
  which implies that $\theta_{\af}(B) \in \eta$. Thus $(\af,\eta)$ is an ultrafilter cycle. 
\end{proof}

\begin{ex}\label{ex5}
We continue Example~\ref{ex1}, Example~\ref{ex2}, Example~\ref{ex3}, and Example~\ref{ex4}. Recall that $x\mapsto\hat{x}$ is a bijection from $X$ to $\widehat{\CB}$. We have that 
\begin{align*}
\widehat{\CR}_1&=\{\hat{x}:x_1=1,\ x_2\in\{1,2\}\},\\
\widehat{\CR}_2&=\{\hat{x}:x_1=2,\ x_2\in\{1,2\}\},\\
\widehat{\CR}_i&=\{\hat{x}:x_1=3\},\\
\widehat{\CR}_j&=\{\hat{x}:x_1=1,\ x_2=3\},
\end{align*}
and that the maps $\widehat{\theta}_a:\widehat{\CR}_a\to \widehat{\CB}$ are given by
\begin{equation*}
\widehat{\theta}_1(\widehat{1x})=\hat{x},\quad
\widehat{\theta}_2(\widehat{2x})=\hat{x},\quad
\widehat{\theta}_i(\widehat{x})=\hat{x},\quad
\widehat{\theta}_j(\widehat{1x})=\hat{x}.
\end{equation*}
Moreover, $(i,\hat{x})$ is an ultrafilter cycle for any $\hat{x}\in \widehat{\CR}_i$, and $(\alpha,\hat{x})$ is an ultrafilter cycle if $\alpha$ is a finite word with letters from $\{1,2\}$ and $x:=\alpha\alpha\dots$ is the infinite word we get by concatenating $\alpha$ with itself infinitely many times.
\end{ex}

\section{Maximal tails}\label{maximal tails}
We now generalize the notion of \emph{maximal tails} from directed graphs (see \cite{BPRS} and \cite{DT}) to our setting. We write $A \geq B$  for $A, B \in \CB$\, if there exists an $\af \in \CL^*$ such that $ B \subseteq \theta_{\af}(A)$. 

\begin{dfn} \label{def:maximal tail}
Let $(\CB, \CL, \theta)$ be a Boolean dynamical system. A non-empty subset $\CT$ of $\CB$ is called a {\it maximal tail}  if
\begin{enumerate}[label=(T\arabic*),start=0]
\item \label{T0} $\emptyset \notin \CT$;
\item \label{T1} if $A \in \CB$ and $\theta_{\af}(A) \in \CT$ for some $\af \in \CL$, then $A \in \CT$;
\item \label{T2} if $A \cup B \in \CT$, then $A \in \CT$ or $B \in \CT$;
\item \label{T3} if $A \in \CT$, $B \in \CB$ and $A \subseteq B$, then $B \in \CT$;
\item \label{T4} if $A \in \CT$ is regular, then there is an $\af \in \CL$ such that $\theta_\af(A) \in \CT$;
\item \label{T5} if $A_1, A_2 \in \CT$, then there exists $C \in \CT$ such that $A_1 \geq C$ and $A_2 \geq C$.
\end{enumerate}
\end{dfn}

\begin{remark} In the above definition, (T1) and (T4) are equivalent to (T1') and (T4'), respectively.
\begin{enumerate}
\item[(T1')] if $A \in \CB$ and $\theta_{\af}(A) \in \CT$ for some $\af \in \CL^*$, then $A \in \CT$;
\item[(T4')] if $A \in \CT$ is regular, then there is an $\af \in \CL^*\setminus\{\emptyset\}$ such that $\theta_\af(A) \in \CT$.
\end{enumerate}
\end{remark}




\begin{lem}\label{remark:1}
If $\CT$ is a maximal tail, then $\CB\setminus\CT$ is a hereditary and saturated ideal of $\CB$.
\end{lem}

\begin{proof}
Suppose $\CT$ is a maximal tail. Then $\CB\setminus\CT$ is non-empty because $\emptyset\notin\CT$. It follows from Property (T2) that if $A,B\in \CB\setminus\CT$, then $A\cup B\in \CB\setminus\CT$; and it follows from Property (T3) that if $A\in \CB\setminus\CT$ and $B\in\CB$, then $A\cap B\in \CB\setminus\CT$. This shows that $\CB\setminus\CT$ is ideal of $\CB$. It follows from Property (T1) that $\CB\setminus\CT$ is hereditary, and from Property (T4) that $\CB\setminus\CT$ is saturated.
\end{proof}

\begin{ex}\label{ex6}
We continue Example~\ref{ex1}, Example~\ref{ex2}, Example~\ref{ex3}, Example~\ref{ex4}, and Example~\ref{ex5}. Recall that $x\mapsto\hat{x}$ is a bijection from $X$ to $\widehat{\CB}$ and that each hereditary and saturated ideal of $\CB$ has the form $\CI_U:=\{A\in\CB:A\subseteq U\}$ for some open subset $U\subseteq X$ for which $\{x\in X:x_1\in\{1,2\}\}$. Using this and the lemma above, it is straightforward to check that $\{\hat{x}:x\in X,\ x_1=3\}$ is the set of all maximal tails of $(\CB,\CL,\theta)$.
\end{ex}

In the example above, each maximal tail is an ultrafilter. In general, a maximal tail is not necessarily a filter (consider for example the case where $\CB$ is the Boolean algebra of subsets of $\mathbb{Z}$, $\CL=\{a\}$ and $\theta_a(A)=\{x+1:x\in A\}$ for $A\in\CB$; then $\CB\setminus\{\emptyset\}$ is a maximal tail, but not a filter).

\begin{prop} \label{prop:cyclic maximal tails}
Let $(\CB, \CL, \theta)$ be a Boolean dynamical system. Suppose $(\alpha,\eta)$ is an ultrafilter cycle and $A\in\eta$ is such that if $\beta\in\CL^*\setminus\{\emptyset\}$, $B\in \CI_A$, and $\theta_\beta(B)\in\eta$, then $B\in\eta$ and $\beta=\alpha^k$ for some $k\in\N$. Then
$$\CT:=\{B\in\CB:\theta_\beta(B)\in\eta\text{ for some }\beta\in\CL^*\}$$
is a maximal tail such that $(\CB/(\CB\setminus\CT),\CL,\theta)$ does not satisfy Condition (L).
\end{prop}

\begin{proof}
It is straightforward to check that $\CT$ satisfies (T0)--(T4). To show it satisfies (T5), we choose $B_1, B_2 \in \CT$. Then there exist $\bt_1, \bt_2 \in \CL^*$ such that $\theta_{\bt_1}(B_1) , \theta_{\bt_2}(B_2) \in \eta.$ Thus $\theta_{\bt_1}(B_1) \cap \theta_{\bt_2}(B_2) \in \eta$, and hence $\theta_\af(\theta_{\bt_1}(B_1) \cap \theta_{\bt_2}(B_2)) \in \eta$ since $(\af, \eta)$ is an ultrafilter cycle. It then follows that $\theta_{\bt_1}(B_1) \cap \theta_{\bt_2}(B_2) \in \CT$ and $B_i \geq \theta_{\bt_1}(B_1) \cap \theta_{\bt_2}(B_2)$ for all $i=1,2$. 

Let $\pi:\CB\to \CB/(\CB\setminus\CT)$ be the quotient map given by $\pi(B)=[B]$. We claim that $(\alpha,[A])$ is a cycle with no exit in $(\CB/(\CB\setminus\CT),\CL,\theta)$. We first show that $$B\setminus\theta_\alpha(B)\notin\CT ~\text{and}~\theta_\alpha(B)\setminus B\notin\CT$$ for any $B\in \CI_A$. For contradiction, suppose $B\setminus\theta_\alpha(B)\in\CT$. Then $\theta_\beta(B)\setminus\theta_{\alpha\beta}(B)\in\eta$ for some $\beta\in\CL^*$. It follows that $\theta_\beta(B)\in\eta$, and thus that $\beta=\alpha^k$ for some $k\in\N$. But then $\theta_{\alpha\beta}(B)=\theta_{\alpha^{k+1}}(B)=\theta_\alpha(\theta_\beta(B))\in\eta$, and $$\emptyset=(\theta_\beta(B)\setminus\theta_{\alpha\beta}(B))\cap \theta_{\alpha\beta}(B)\in\eta,$$ which is not the case. Thus, we must have that $B\setminus\theta_\alpha(B)\notin\CT$. Similarly, if $\theta_\alpha(B)\setminus B\in\CT$, then $\theta_{\alpha\beta}(B)\setminus \theta_\beta(B)\in\eta$ for some $\beta\in\CL^*$, and then $\theta_{\alpha\beta}(B)\in\eta$, which implies that $B\in\eta$, and thus that $\theta_\beta(B)\in\eta$ and $$\emptyset=(\theta_{\alpha\beta}(B)\setminus\theta_\beta(B))\cap \theta_\beta(B)\in\eta,$$ which is not the case. Thus, we must have that $\theta_\alpha(B)\setminus B\notin\CT$.

Now for any $[B] \subseteq [A]$, we see that 
$$\theta_\af([B])=\theta_\af([A\cap B])=[A \cap B]=[B].$$
This shows that $(\af, [A])$ is a cycle. To show that  $(\af, [A])$ has no exit, assume to the contrary that there are $t \leq |\af|$ and $ [\emptyset] \neq [B] \subseteq [\theta_{\af_1 \cdots \af_t}(A)]$ such that $\Delta_{[B]} :=\{\bt \in \CL: [\theta_{\bt}(B)] \neq [\emptyset]\} \neq \{\af_{t+1}\}$. Then $[\theta_{\af_1 \cdots \af_t \bt}(A)] \neq [\emptyset]$ for some $\bt \neq \af_{t+1}$, and then $\theta_{\af_1 \cdots \af_t \bt}(A) \in \CT$. Thus there is a $\gm \in \CL^*$ such that $\theta_{\gm}(\theta_{\af_1 \cdots \af_t \bt}(A)) =\theta_{\af_1 \cdots \af_t \bt \gm}(A)\in \eta$, and hence $\af_1 \cdots \af_t \bt \gm =\af^k$ for some $k \in \N$, which means that $\bt=\af_{t+1}$, a contradiction.

\end{proof}

\begin{dfn} \label{def:cyclic tail}
Let $(\CB, \CL, \theta)$ be a Boolean dynamical system. A maximal tail $\CT$ of $(\CB, \CL, \theta)$ is \emph{cyclic} (cf. \cite[p. 112]{CS}) if there is an ultrafilter cycle $(\alpha,\eta)$ such that $\CT:=\{B\in\CB:\theta_\beta(B)\in\eta\text{ for some }\beta\in\CL^*\}$ and an $A\in\eta$ such that if $\beta\in\CL^*\setminus\{\emptyset\}$, $B\in \CI_A$, and $\theta_\beta(B)\in\eta$, then $B\in\eta$ and $\beta=\alpha^k$ for some $k\in\N$.
\end{dfn}

\begin{ex}\label{ex6.5}
We continue Example~\ref{ex1}, Example~\ref{ex2}, Example~\ref{ex3}, Example~\ref{ex4}, Example~\ref{ex5}, and Example~\ref{ex6}. Recall that $x\mapsto\hat{x}$ is a bijection from $X$ to $\widehat{\CB}$ and that $\{\hat{x}:x\in X,\ x_1=3\}$ is the set of all maximal tails of $(\CB,\CL,\theta)$. We next show that these maximal tails are all cyclic.

To see this, let $x$ be an element of $X$ such that $x_1=3$. We saw in Example~\ref{ex5} that $(i,\hat{x})$ is then an ultrafilter cycle. Moreover, 
$$
\hat{x}=\{B\in\CB:\theta_\beta(B)\in\hat{x}\text{ for some }\beta\in\CL^*\}.
$$ 
Let $A:=\{y\in X:y_1=3\}$. Suppose $\beta\in\CL^*\setminus\{\emptyset\}$, $B\in\CI_A$, and $\theta_\beta(B)\in\hat{x}$. Since $\theta_\alpha(B)\notin\hat{x}$ for $\alpha\in\{1,2,t\}$ (see Example~\ref{ex2}), it follows that $\beta=i^k$ for some $k\in\N$, and thus that $B=\theta_\beta(B)\in\hat{x}$. We thus have that $\hat{x}$ is a cyclic maximal tail.
\end{ex}

Next, we show a converse to Proposition~\ref{prop:cyclic maximal tails}.

\begin{prop} \label{prop:K}
Let $(\CB, \CL, \theta)$ be a Boolean dynamical system. Suppose $\CH$ is a hereditary and saturated ideal of $\CB$ such that $(\CB/\CH,\CL,\theta)$ does not satisfy Condition (L). Then there is a cyclic maximal tail $\CT$ such that $\CT\cap\CH=\emptyset$.
\end{prop}




Before we give the proof of Proposition~\ref{prop:K}, we first introduce the following two lemmas which we shall use in the proof of Proposition~\ref{prop:K} and again later in the paper.

\begin{lem}\label{sub-cycle} 
Let $(\CB,\CL,\theta)$ be a Boolean dynamical system and suppose $(\af, A)$ is a cycle with no exits.  Fix $1 \leq k < |\af|$ and let $\af'=\af_{[1,k]}$.  Suppose that $B \cap \theta_{\af'}(B) \neq \emptyset$ for every non-empty $B \in \CB$ with $B \subseteq A$, and let $A':=A \cap \theta_{\af'}(A)$. Then $(\af', A')$ is a cycle with no exit.
\end{lem}

\begin{proof} 
It suffices to show that $ B=\theta_{\af'}(B)$ for all $B \subseteq A'$. Choose a non-empty  $B \in \CB$ such that  $B \subseteq A'$. 
Then
$$ (B \setminus \theta_{\af'}(B)) \cap \theta_{\af'}(B \setminus \theta_{\af'}(B))= \emptyset.$$
From the assumption that $B' \cap \theta_{\af'}(B') \neq \emptyset$ for every non-empty $ B' \subseteq A$, we thus get that $B \setminus \theta_{\af'}(B) = \emptyset$. It follows that $ B \subseteq \theta_{\af'}(B)$. 

Since $(\af, A')$ is a cycle with no exit and $B' \cap \theta_{\af'}(B') \neq \emptyset$ for every non-empty $ B'$ with $B' \subseteq A$, it follows that $\af_{k+i}=\af_i$ for every $i$, where the indices are computed modulo $|\alpha|$. We thus have that $(\alpha')^{|\alpha|}=\alpha^k$, and hence $\theta_{\alpha'}(B)\subseteq \theta_{(\alpha')^2}(B)\subseteq\dots\subseteq \theta_{(\alpha')^{|\alpha|}}(B)=B$. So, $B = \theta_{\af'}(B)$. 
\end{proof}

\begin{lem}\label{no intersection} 
Let $(\CB,\CL,\theta)$ be a Boolean dynamical system and suppose $(\af, A)$ is a cycle with no exits. Then there exist $ 1 \leq j \leq  |\af|$ and $B \subseteq A$ such that $(\beta, B)$ is a cycle with no exits and $B \cap \theta_{\beta_{[1,k]}}(B) = \emptyset$ for all $1 \leq k < j$, where $\beta=\af_{[1,j]}$.
\end{lem}

 \begin{proof} 
 Let 
 \begin{equation*}
 j:=\min\{k: 1\le k\le |\af|\text{ and there exists } B\in\mathcal{I}_A\text{ such that } (\af_{[1,k]},B)\text{ is a cycle}\}.
 \end{equation*}
 Choose $A_j\in\mathcal{I}_A$ such that $(\af_{[1,j]},A_j)$ is a cycle. By applying Lemma~\ref{sub-cycle} inductively on $(\af_{[1,j]},A_{k+1})$, where $k=j-1,j-2,\dots,1$, we see that there are non-empty $A_j\supseteq A_{j-1}\supseteq \dots\supseteq A_1$ in $\mathcal{B}$ such that $A_k\cap \theta_{\alpha_{[1,k]}}(A_j)=\emptyset$ for each $1\le k< j$. Then $B:=A_1$ has the desired properties.
  \end{proof}  

\begin{proof}[Proof of Proposition~\ref{prop:K}]
It follows from Lemma \ref{no intersection} that there is a cycle $(\af,[A])$ with no exits in $(\CB/\CH,\CL,\theta)$  such that $[A]\cap\theta_{\alpha_{[1,k]}}([A])=\emptyset$ for $1\le k<|\alpha|$. Let $\eta'$ be an ultrafilter in $\CB/\CH$ such that $[A]\in \eta'$. By Lemma \ref{ultra filter cycle lemma}(3), we then have that $(\af, \eta')$ is ultrafilter cycle in $(\CB/\CH,\CL,\theta)$.    Let $$\eta:=\{B\in\CB:[B]\in \eta'\}.$$ Then $\eta$ is an ultrafilter in $\CB$ and $A \in \eta$.

If $B\in\eta$, then $[\theta_\alpha(B)]=\theta_\alpha([B])\in \eta'$, and hence $\theta_\alpha(B)\in\eta$. This shows that $(\alpha,\eta)$ is an ultrafilter cycle.
Suppose $\beta\in\CL^*\setminus\{\emptyset\}$, $B\in \CI_A$, and $\theta_\beta(B)\in\eta$. Then $\theta_\beta([B])=[\theta_\beta(B)]\in \eta'$. Since also $[A]\in\eta'$, it follows that $[A]\cap \theta_\beta([B])\neq\emptyset$. Since $[B] \subseteq [A]$, we thus have $$  \emptyset \neq  [A]\cap \theta_\beta([B]) \subseteq  [A]\cap\theta_\beta([A]).$$
Using that $(\af,[A])$ is a cycle with no exits and $[A]\cap\theta_{\alpha_{[1,k]}}([A])=\emptyset$ for $1\le k<|\alpha|$, we see that $\beta=\alpha^k$ for some $k\in\N$. Thus $$[\theta_\beta(B)] =\theta_\beta([B])=[B]\in \eta'.$$ 
This shows that $B\in\eta$.

We thus have that $\CT:=\{B\in\CB:\theta_\beta(B)\in\eta\text{ for some }\beta\in\CL^*\}$ is a cyclic maximal tail. Suppose $B\in \CT\cap\CH$. Since $B\in\CT$, we have that $\theta_\beta([B])\in\eta'$ for some $\beta\in\CL^*$. On the other hand, since $B\in\CH$ we also have that $[B]=\emptyset$, a contradiction. Thus, $\CT\cap\CH=\emptyset$. 
\end{proof}

\section{Condition (K)}\label{Condition (K)}

Motivated by Proposition~\ref{prop:cyclic maximal tails} and Proposition~\ref{prop:K} we make the following definition.

\begin{dfn}\label{def:cond(K)}
We say that a Boolean dynamical system $(\CB, \CL, \theta)$ satisfies Condition (K) if there is no pair $((\alpha,\eta),A)$ where $(\alpha,\eta)$ is an ultrafilter cycle and $A\in\eta$ such that if $\beta\in\CL^*\setminus\{\emptyset\}$, $B\in \CI_A$, and $\theta_\beta(B)\in\eta$, then $B\in\eta$ and $\beta=\alpha^k$ for some $k\in\N$.
\end{dfn}

\begin{remark}
It immediately follows from Proposition~\ref{prop:cyclic maximal tails} and Definition~\ref{def:cyclic tail} that a Boolean dynamical system satisfies Condition (K) if and only if it has no cyclic maximal tails.
\end{remark}

\begin{ex}\label{ex7}
We continue Example~\ref{ex1}, Example~\ref{ex2}, Example~\ref{ex3}, Example~\ref{ex4}, Example~\ref{ex5}, Example~\ref{ex6}, and Example~\ref{ex6.5}. We saw in Example~\ref{ex6.5} that $(\CB, \CL, \theta)$ has cyclic maximal tail. It therefore does not satisfy Condition (K). 


\end{ex}

\vskip 0.5pc 
We shall see in Theorem ~\ref{equivalent:(K)} that a locally finite Boolean dynamical system with countable $\CB$ and ${\CL}$ satisfies Condition (K) if and only if every ideal in $C^*(\CB, \CL, \theta)$ is gauge-invariant, but we shall first look at how our Condition (K) for Boolean dynamical systems is related to Condition (K) for directed graphs. This will also provides us with examples of Boolean dynamical systems that satisfy Condition (K).

We refer the reader to \cite{BHRS, BPRS, DT, KPR, KPRR, R} among others for the definitions of directed graphs and their $C^*$-algebras. Let $E=(E^0,E^1, r,s)$ be a directed graph. We shall define two Boolean dynamical systems $(\mathcal{B}_E,\mathcal{L}_E,\theta_E)$ and $(\mathcal{B}_{\partial E},\mathcal{L}_{\partial E},\theta_{\partial E})$, both of which have the same $C^*$-algebra as $E$.

We let $\CL_E:=E^1$, let $\CB_E$ be the set of finite subsets of $E^0$, and define for each $e\in\CL_E$ a map $(\theta_E)_e:\CB_E\to\CB_E$ by
\begin{equation*}
(\theta_E)_e(A)=
\begin{cases}
\{r(e)\}&\text{if }s(e)\in A,\\
\emptyset&\text{if }s(e)\notin A.
\end{cases}
\end{equation*}
Then $(\mathcal{B}_E,\mathcal{L}_E,\theta_E)$ is a Boolean dynamical system. Let $\{p_A,\ s_e:A\in \CB_E,\ e\in\CL_E\}$ be a universal Cuntz--Krieger representation of $(\mathcal{B}_E,\mathcal{L}_E,\theta_E)$. Then $\{s_e, p_{\{v\}}: e \in E^1, v \in E^0\}$ is a Cuntz-Krieger $E$-family. It therefore follows from the universal property of $C^*(E)$ that there is $*$-homomorphism $\phi: C^*(E) \to C^*(\mathcal{B}_E,\mathcal{L}_E,\theta_E)$ that maps $s_e$ to  $s_e$ and  $p_v$ to  $p_{\{v\}}$. Since $p_A=\sum_{v\in A}p_{\{v\}}$ for $A\in\CB_E$, it follows that $\phi$ is onto, and it is injective by the gauge-invariant  uniqueness theorem \cite[Theorem 2.1]{BHRS}, Remark~\ref{basics}(1) and the fact that $p_A\ne 0$ for $A\ne\emptyset$ (the latter follows for instance by an application of Theorem~\ref{CK uniqueness thm} to the identity map of $C^*(\mathcal{B}_E,\mathcal{L}_E,\theta_E)$). Thus, $C^*(\mathcal{B}_E,\mathcal{L}_E,\theta_E)$ is isomorphic to $C^*(E)$.


 Moreover, the map
\begin{equation*}
v\mapsto \hat{v}:=\{A\in\CB_E:v\in A\}
\end{equation*}
is a bijection between $E^0$ and $\widehat{\CB}_E$ such that 
\begin{equation*}
\hat{\theta}_e(\hat{v})=
	\begin{cases}
	\widehat{s(e)}&\text{if }r(e)=v,\\
	\emptyset&\text{if }r(e)\ne v,
	\end{cases}
\end{equation*}
for $v\in E^0$ and $e\in\CL_E$.

Let $E^0_{\sing}:=\{v\in E^0: s^{-1}(v)\text{ is empty or infinite}\}$ be the set of singular vertices. For $n\in\N$, let 
$$E^n:=\{x_1x_2 \dots x_n\in (E^1)^n: r(x_i)=s(x_{i+1}) ~\text{for all}~ i\},$$
and define $E^*=\cup_{n \geq 0} E^n$ to be the set of all finite paths, where we regard a vertex in $E^0$ as a path of length 0. 
Similarly, we let 
$$E^\infty:=\{x_1x_2 \dots \in (E^1)^\N: r(x_i)=s(x_{i+1}) ~\text{for all}~ i\}.$$
The range and source maps $r$ and $s$ extend to $E^*$ in the obvious way, and $s$ extends to $E^\infty$. We write $|u|=n$ if $u\in E^n$. The boundary path space of $E$ is the space
$$\partial E:=E^\infty\cup\{u\in E^*:r(u)\in E^0_{\sing}\}$$
equipped with the topology for which the generalized cylinder sets
\begin{multline*}
Z(\af\setminus F):=\{ x \in \partial E: |\af|\le |x|,\ s(x)=\af\text{ if }|\af|=0,\\ x_1=\af_1, \dots, x_{|\af|}=\af_{|\af|} \text{ if }|\af|>0,\ \text{and either }|x|=|\af|\text{ or }x_{|\af|+1}\notin F \}
\end{multline*}
parametrized by pairs $(\af,F)$ where $\af \in E^*$ and $F$ is a finite subset of $s^{-1}(r(\af))$, form a basis of compact open sets, so that  $\partial E$ is locally compact Hausdorff space  (\cite[Theorems 2.1 and 2.2]{Web}).

We let $\CL_{\partial E}:=E^1$, define $\mathcal{B}_{\partial E}$ the set of compact open subsets of the boundary path space $\partial E$, and define for each $e\in\mathcal{L}_{\partial E}$ a Boolean homomorphism $(\theta_{\partial E})_e:\mathcal{B}_{\partial E}\to\mathcal{B}_{\partial E}$ by setting 
$$(\theta_{\partial E})_e(A):=\{x\in \partial E:ex\in A\}.$$ 
Then $(\mathcal{B}_{\partial E},\mathcal{L}_{\partial E},\theta_{\partial E})$ is a Boolean dynamical system. 

Let $\{p_A,\ s_e:A\in \CB_{\partial E},\ e\in\CL_{\partial E}\}$ be a universal Cuntz--Krieger representation of $(\mathcal{B}_{\partial E},\mathcal{L}_{\partial E},\theta_{\partial E})$. Then $\{s_e, p_{Z(v\setminus\emptyset)}: e \in E^1, v \in E^0\}$ is a Cuntz-Krieger $E$-family. It therefore follows from the universal property of $C^*(E)$ that there is $*$-homomorphism $\phi: C^*(E) \to C^*(\mathcal{B}_{\partial E},\mathcal{L}_{\partial E},\theta_{\partial E})$ that maps $s_e$ to  $s_e$ and  $p_v$ to  $p_{Z(v\setminus\emptyset)}$. 

It follows from the properties of $\{p_A,\ s_e:A\in \CB_{\partial E},\ e\in\CL_{\partial E}\}$ that if $v\in E^0$ and $F$ is a finite subset of $s^{-1}(v)$, then $p_{Z(v\setminus F)}=p_{Z(v\setminus\emptyset)}-\sum_{e\in F}s_es_e^*$, and if $\alpha\in E^*\setminus E^0$ and $F$ is a finite subset of $s^{-1}(r(\alpha))$, then $p_{Z(\alpha\setminus F)}=s_\alpha s_\alpha^*-\sum_{e\in F}s_{\alpha e}s_{\alpha e}^*$. Since $\CB_{\partial E}$ is generated by elements of the form $Z(\alpha\setminus F)$ where $\alpha\in E^*$ and $F$ is a finite subset of $s^{-1}(r(\alpha))$, it follows that $\{p_A,\ s_e:A\in \CB_{\partial E},\ e\in\CL_{\partial E}\}$ is in the image of $\phi$, and thus that $\phi$ is surjective. It follows from the gauge-invariant  uniqueness theorem \cite[Theorem 2.1]{BHRS}, Remark~\ref{basics}(1) and the fact that $p_A\ne 0$ for $A\ne\emptyset$ (the latter follows for instance by an application of Theorem~\ref{CK uniqueness thm} to the identity map of $C^*(\mathcal{B}_{\partial E},\mathcal{L}_{\partial E},\theta_{\partial E})$) that $\phi$ is injective. Thus, $C^*(\mathcal{B}_{\partial E},\mathcal{L}_{\partial E},\theta_{\partial E})$ is isomorphic to $C^*(E)$.


 Moreover, the map
\begin{equation*}
x \in \partial E \mapsto \hat{x}:=\{A \in \CB_{\partial E} : x \in A\}
\end{equation*}
is a bijection between $\partial E$ and $\widehat{\CB}_{\partial E}$. Note that if $\bt \in \CL_{\partial E}^*\setminus\{\emptyset\}=E^*\setminus E^0$, then
$$(\hat{\theta}_{\partial E})_\beta(\hat{x})= 
	\left\{ \begin{array}{ll}
        \widehat{\beta x} &  ~~\hbox{if\ }  s(x)=r(\beta), \\
        \emptyset &  ~~\hbox{if\ }  s(x)\ne r(\beta).
    \end{array}\right.
$$

Recall that the graph $E$ satisfies Condition (K) if and only if whenever $v\in E^0$, $\alpha\in E^*\setminus E^0$, and $s(\alpha)=r(\alpha)=v$, then there is a $\beta\in E^*\setminus E^0$ such that $s(\beta)=r(\beta)=v$ and $\beta\ne\alpha^k$  for  all $k\in\N$. 

\begin{remark}\label{graph-BDS} 
Let $E$, $(\mathcal{B}_E,\mathcal{L}_E,\theta_E)$, and $(\mathcal{B}_{\partial E},\mathcal{L}_{\partial E},\theta_{\partial E})$ be as above.
\begin{enumerate}
\item[(a)] If $v\in E^0$ and $\alpha\in\CL_E^*$, then $(\alpha,\hat{v})$ is an ultrafilter cycle in $(\mathcal{B}_E,\mathcal{L}_E,\theta_E)$ if and only if $s(\alpha)=r(\alpha)=v$.

\vskip 0.2pc
\item[(b)] If $x\in\partial E$ and $\alpha\in \mathcal{L}_{\partial E}^*$, then $(\alpha,\hat{x})$ is an ultrafilter cycle in $(\mathcal{B}_{\partial E},\mathcal{L}_{\partial E},\theta_{\partial E})$ if and only if $x=\alpha^\infty$.

\vskip 0.2pc
\item[(c)] $(\mathcal{B}_E,\mathcal{L}_E,\theta_E)$ satisfies Condition (K) if and only if $E$ satisfies Condition (K).

\vskip 0.2pc
\item[(d)] $(\mathcal{B}_{\partial E},\mathcal{L}_{\partial E},\theta_{\partial E})$ satisfies Condition (K) if and only if $E$ satisfies Condition (K).
\vskip 0.5pc
\end{enumerate}
\end{remark}

\begin{proof}
(a) and (b) easily follow from Lemma~\ref{ultra filter cycle lemma}(2).

(c): Suppose first that $(\mathcal{B}_E,\mathcal{L}_E,\theta_E)$ satisfies Condition (K), and assume that $v\in E^0$, $\alpha\in E^*\setminus E^0$, and $s(\alpha)=r(\alpha)=v$. Then $(\alpha,\hat{v})$ is an ultrafilter cycle and $\{v\}\in\hat{v}$. Since we are assuming that $(\mathcal{B}_E,\mathcal{L}_E,\theta_E)$ satisfies Condition (K), it therefore follows that there is a $B\in\CI_{\{v\}}$ and a $\beta\in\CL^*\setminus\{\emptyset\}$ such that $(\theta_E)_\beta(B)\in\hat{v}$, and either $B\notin\hat{v}$ or $\beta\ne\alpha^k$ for any $k\in\N$. The conditions $B\in\CI_{\{v\}}$ and $(\theta_E)_\beta(B)\in\hat{v}$ together imply that $B=\{v\}$ and $s(\beta)=r(\beta)=v$. We therefore have that $\beta\ne\alpha^k$ for any $k\in\N$ which shows that $E$ satisfies Condition (K).

Conversely, suppose $E$ satisfies Condition (K), and assume that $(\alpha,\hat{v})$ is an ultrafilter cycle in $(\mathcal{B}_E,\mathcal{L}_E,\theta_E)$ and $A\in\hat{v}$. Then $s(\alpha)=r(\alpha)=v$. Since we are assuming that $E$ satisfies Condition (K), it follows that there is a $\beta\in E^*\setminus E^0$ such that $s(\beta)=r(\beta)=v$ and $\beta\ne\alpha^k$ for any $k\in\N$. Then $\beta\in\CL_E^*\setminus\{\emptyset\}$, $\{v\}\in\CI_A$, $(\theta_E)_\beta(\{v\})=\{v\}\in\hat{v}$, and $\beta\ne\alpha^k$ for any $k\in\N$; which shows that $(\mathcal{B}_E,\mathcal{L}_E,\theta_E)$ satisfies Condition (K).

(d): Suppose first that $(\mathcal{B}_{\partial E},\mathcal{L}_{\partial E},\theta_{\partial E})$ satisfies Condition (K), and assume that $v\in E^0$, $\alpha\in E^*\setminus E^0$, and $s(\alpha)=r(\alpha)=v$. Then $x:=\alpha^\infty\in\partial E$ and $(\alpha,\hat{x})$ is an ultrafilter cycle in $(\mathcal{B}_{\partial E},\mathcal{L}_{\partial E},\theta_{\partial E})$. Let $A:=Z(\alpha\setminus\emptyset)$. Then $A\in\hat{x}$. Since we are assuming that $(\mathcal{B}_{\partial E},\mathcal{L}_{\partial E},\theta_{\partial E})$ satisfies Condition (K), it therefore follows that there is a $B\in\CI_A$ and a $\beta\in\CL^*\setminus\{\emptyset\}$ such that $(\theta_{\partial E})_\beta(B)\in\hat{x}$, and either $B\notin\hat{x}$ or $\beta\ne\alpha^k$ for any $k\in\N$. It follows from $(\theta_{\partial E})_\beta(B)\in\hat{x}$ that $\beta x\in B$, and thus that $r(\beta)=s(x)=s(\alpha)=v$. Since $B\in\CI_A$, we also have that $\beta x\in B\subseteq A=Z(\alpha\setminus\emptyset)$ and thus $s(\beta)=s(\alpha)=v$. If $\beta=\alpha^k$ for some $k\in\N$, then $x=\beta x\in B$, so that cannot be the case. Thus, $\beta\ne\alpha^k$ for any $k\in\N$. This shows that $E$ satisfies Condition (K).

Conversely, suppose $E$ satisfies Condition (K), and assume that $(\alpha,\hat{x})$ is an ultrafilter cycle in $(\mathcal{B}_{\partial E},\mathcal{L}_{\partial E},\theta_{\partial E})$ and $A\in\hat{x}$. Then $s(\alpha)=r(\alpha)$ and $x=\alpha^\infty$. Since we are assuming that $E$ satisfies Condition (K), it follows that there is a $\beta\in E^*\setminus E^0$ such that $s(\beta)=r(\beta)=v$ and $\beta\ne\alpha^k$ for any $k\in\N$. Since $A\in\hat{x}$, we have that $A$ is an open neighborhood of $x$ in $\partial E$. There is therefore an $n\in\N$ such that $Z(\alpha^n\setminus\emptyset)\subseteq A$. Let $\gamma:=\alpha^n\beta$ and $B:=Z(\gm \setminus\emptyset)$. Then $\gamma\in\CL_{\partial E}^*\setminus\{\emptyset\}$, $B\in\CI_A$, $(\theta_{\partial E})_\gamma(B)\in\hat{x}$, and $\gamma\ne\alpha^k$ for any $k\in\N$. This shows that $(\mathcal{B}_{\partial E},\mathcal{L}_{\partial E},\theta_{\partial E})$ satisfies Condition (K).
\end{proof}

Remark~\ref{graph-BDS} might make one think that Condition (K) for Boolean dynamical systems would be equivalent to one of the two conditions in the next lemma. However, we shall in Proposition~\ref{prop:graphK} and Remark~\ref{remark:weak} see that the two conditions in Lemma~\ref{strong K} are, in general, strictly stronger than Condition (K).

\begin{lem}\label{strong K} Let  $(\CB, \CL, \theta)$ be a Boolean dynamical system. The following are equivalent. 
\begin{enumerate}
\item If $(\alpha,\eta)$ is an ultrafilter cycle, then there exists $\bt \in \CL^*$ such that $(\bt, \eta)$ is an ultrafilter cycle and $\bt \neq \af^k$ for all $k \in \N$.
\item If $\af \in \CL^*\setminus\{\emptyset\}$, $\eta\in \widehat{\CB}$, and $\eta=\widehat{\theta}_{\af}(\eta)$, then there exists $\bt \in \CL^*\setminus\{\emptyset\}$ such that $\eta=\widehat{\theta}_{\bt}(\eta)$ and $\bt \neq \af^k$ for all $k \in \N$.
\end{enumerate}
\end{lem}

\begin{proof} Follows from Lemma~\ref{ultra filter cycle lemma}(2).
\end{proof}

\begin{prop} \label{prop:graphK}
Let  $(\CB, \CL, \theta)$ be a Boolean dynamical system. If $(\CB, \CL,\theta)$ satisfies condition (1) in Lemma~\ref{strong K}, then it satisfies Condition (K).
\end{prop}
\begin{proof}  
Assume (1) in Lemma~\ref{strong K} holds and let $(\alpha,\eta)$ be an ultrafilter cycle and $A\in\eta$. Then there is a $\bt \in \CL^*\setminus\{\emptyset\}$ such that $(\bt, \eta)$ is an ultrafilter cycle and $\bt \neq \af^k$ for all $k \in \N$. Since $(\beta,\eta)$ is an ultrafilter cycle and $A\in\eta$, it follows that $\theta_\beta(A)\in\eta$. This shows that $(\CB, \CL,\theta)$ satisfies Condition (K).
\end{proof}
  
   
\begin{remark} \label{remark:weak}
If a directed graph $E$ satisfies Condition (K), then the Boolean dynamical system $(\mathcal{B}_{\partial E},\mathcal{L}_{\partial E},\theta_{\partial E})$ satisfies Condition (K) by Remark~\ref{graph-BDS}(d). But, it follows from Remark \ref{graph-BDS}(b) that condition (1) in Lemma~\ref{strong K} cannot hold for $(\mathcal{B}_{\partial E},\mathcal{L}_{\partial E},\theta_{\partial E})$. Thus,  
Condition (K) does not imply condition (1) in Lemma~\ref{strong K} in general.
\end{remark}

\section{Boolean dynamical systems for which all ideals are gauge-invariant}\label{gau-inv-ideal}
In this section, we show that 
a Boolean dynamical system $(\CB,\CL,\theta)$ satisfies Condition (K) if and only if the quotient Boolean dynamical system $(\CB/\CH, \CL, \theta)$ satisfies Condition (L) for every hereditary saturated ideal $\CH$ of $\CB$. 
We also show that each of them is a necessary condition to that every ideal of $C^*(\CB, \CL, \theta)$ is gauge-invariant, and that  if  moreover $(\CB, \CL, \theta)$ is locally finite and has countable $\CB$ and $\CL$,  each of them is also a sufficient condition to that  every  ideal of $C^*(\CB, \CL, \theta)$ is gauge-invariant.

We start with  two technical results.

\begin{lem}\label{no exit:simple cycle} Let $(\CB,\CL,\theta)$ be a Boolean dynamical system and suppose $(\af, A)$ is a cycle with no exits. Then we have $\theta_{\af_{[1,i]}}(A) \cap \theta_{\af_{[1,j]}}(A) = \emptyset ~~\text{for all}~~  1 \leq i < j \leq |\af|.$
\end{lem}
   
\begin{proof} Let $(\af, A)$ be a cycle with no exits. By Lemma~\ref{no intersection}, we may assume that  $A\cap \theta_{\alpha_{[1,j]}}(A)=\emptyset$ for $1\le j<|\alpha|$. Let $n:=|\alpha|$. We claim that then 
$$\theta_{\af_{[1,i]}}(A) \cap \theta_{\af_{[1,j]}}(A) = \emptyset ~~\text{for all}~~  1 \leq i < j \leq n.$$
To see that the claim holds, suppose $\theta_{\af_{[1,i]}}(A) \cap \theta_{\af_{[1,j]}}(A) \ne \emptyset$ for some $1 \leq i < n$. Since $(\alpha,A)$ has no exit, we then have that $\alpha_{i+k}=\alpha_{j+k}$ for any $k$ (where the indices are computed module $n$). Thus, 
$$A\cap \theta_{\alpha_{[1,j-i]}}(A)= \theta_{\alpha_{i+1}\alpha_{i+2}\dots \alpha_n}(\theta_{\af_{[1,i]}}(A) \cap \theta_{\af_{[1,j]}}(A))\ne  \emptyset$$ 
because $\alpha_{i+1}\in\Delta_{\theta_{\af_{[1,i]}}(A) \cap \theta_{\af_{[1,j]}}(A)}$ and $\alpha_{i+k+1}\in\Delta_{\theta_{\alpha_{[i+1, i+k]}}(\theta_{\af_{[1,i]}}(A) \cap \theta_{\af_{[1,j]}}(A))}$ for $1\le k< n-i$. But this contradicts the assumption that $A\cap \theta_{\alpha_{[1,j]}}(A)=\emptyset$.
\end{proof}
   \vskip 1pc
\begin{prop}\label{prop2:cyclic maximal tails}
Let $(\CB, \CL, \theta)$ be a Boolean dynamical system. Suppose $(\CB, \CL, \theta)$ has a cyclic maximal tail $\CT$. Then $C^*(\CB/(\CB\setminus\CT),\CL,\theta)$ contains an ideal that is not gauge-invariant, and there is a $B\in\CT$ such that $p_{[B]}C^*(\CB/(\CB\setminus\CT),\CL,\theta)p_{[B]}$ is isomorphic to $M_n(C(\mathbb{T}))$ for some $n\in\mathbb{N}$.  
\end{prop}

\begin{proof}
Choose a cyclic maximal tail $\mathcal T$ in $(\CB,\CL,\theta)$. Then there is an ultrafilter cycle $(\alpha,\eta)$ such that $\mathcal{T}=\{B\in\mathcal{B}:\theta_\beta(B)\in\eta\text{ for some }\beta\in\mathcal{L}^*\}$ and an $A\in\eta$ such that if $\beta\in\mathcal{L}^*\setminus\{\emptyset\}$, $B\in\mathcal{I}_A$ and $\theta_\beta(B)\in\eta$, then $B\in\eta$ and $\beta=\alpha^k$ for some $k\in\mathbb{N}$.
By Remark \ref{remark:1}, $\CB\setminus\CT$ is a hereditary saturated ideal of $\CB$. We  first show that for any $B \in \CI_A$, we have 
\begin{align}\label{minimal}\text{either} \hskip1pc A \setminus B \notin \CT  \hskip1pc \text{or} \hskip1pc B \notin \CT.
\end{align} 
(or equivalently, either  $A \setminus B \in \CB\setminus\CT  $ or $B \in \CB\setminus\CT$.)
Since $A=B \cup (A \setminus B) \in \eta$, either $B \in \eta$ or $A \setminus B \in \eta$. 
First, if $B \in \eta$, then $A \setminus B \notin \eta$. We then have $\theta_\bt(A \setminus B) \notin \eta$ for all $\bt \in \CL^*.$ If not, $\theta_\bt(A \setminus B) \in \eta$ for some $\bt \in \CL^*$, and thus  $A \setminus B \in \eta$. This is not the case.
 Thus $A \setminus B \notin \CT$. Second, if $ A \setminus B \in \eta$, it follows that 
 $\theta_\bt(B) \notin \eta$ for all $\bt \in \CL^*$ 
with the same reason. Thus, $B \notin \CT$.

Put $B:=\cup_{k=1}^n \theta_{\af_{[1,k]}}(A)$ where $n=|\alpha|$. We now claim that 
$$p_{[B]}C^*(\CB/(\CB\setminus\CT), \CL,\theta)p_{[B]} \cong C(\mathbb{T}) \otimes M_n.$$
We proved that $(\af,[A])$ is a cycle with no exit in $(\CB/(\CB\setminus\CT), \CL,\theta)$ in the proof of Proposition \ref{prop:cyclic maximal tails}.  By  (\ref{minimal}), we see that $[A]$ is minimal in the sense that for any non-empty $[C] \in \CB/(\CB\setminus\CT)$, either 
$[A \cap C]=[A]$ or $[A\cap C]=[\emptyset]$. 
We also have by Lemma \ref{no exit:simple cycle} that 
 \begin{align}\label{simple cycle}[\theta_{\af_{[1,i]}}(A)] \cap [\theta_{\af_{[1,j]}}(A)]= \emptyset ~~\text{for all}~~  1 \leq i < j \leq n.
\end{align}

Then, for $s_{\mu}p_{[C]}s_{\nu}^* \in C^*(\CB/(\CB\setminus\CT), \CL,\theta )$, if
$$p_{[B]}(s_{\mu}p_{[C]}s_{\nu}^*)p_{[B]} = 
s_{\mu}p_{[\theta_\mu(B)]\cap [C] \cap [\theta_\nu(B)] }s_{\nu}^* \neq 0,$$ 
then 
$[\theta_\mu(B)] \cap [\theta_\nu(B)]
\neq \emptyset.$ Thus  $[\theta_\mu(B)] \neq \emptyset$ and  $ [\theta_\nu(B)]\neq \emptyset$, and hence we see that the paths $\mu$, $\nu$ are of the form 
$$\mu=\af_{[i ,n]}\af^l\af_{[1,k]},\ 
\nu=\af_{[j,n]}\af^m\af_{[1,k']}$$ 
for some $i,j,l,m \geq 0$ and $1 \leq k, k' \leq n$ since $(\af,[A])$ is a cycle with no exit. Then $ \emptyset \neq[\theta_\mu(B)] \cap [\theta_\nu(B)] = [\theta_{\af_{[1, i-1]}\mu}(A)] \cap [\theta_{\af_{[1,j-1]}\nu}(A)]=[\theta_{\af_{[1,k]}}(A)]\cap [\theta_{\af_{[1,k']}}(A)]
$. Thus by (\ref{simple cycle}), $k=k'$. 
It then    follows  that 
\begin{align*} 
s_{\mu}p_{[\theta_\mu(B)]\cap [C] \cap [\theta_\nu(B)] }s_{\nu}^*
& =s_{\af_{[i,n]}\af^l\af_{[1,k]}}
  p_{[\theta_{\af_{[1,k]}}(A) \cap C]}s_{\af_{[j,n]}\af^m\af_{[1,k]}}^* \\
& = s_{\af_{[i,n]}\af^l\af_{[1,k]}}s_{\af_{k+1}}p_{[\theta_{\af_{[1,k+1]}}(A) \cap \theta_{\af_{k+1}}(C)]}s_{\af_{k+1}}^*s_{\af_{[j,n]}\af^m\af_{[1,k]}}^* \\
& \hskip 0.5pc \vdots \\
&=s_{\af_{[i,n]}\af^l\af_{[1,n]}}p_{[\theta_{\af_{[1,n]}}(A) \cap \theta_{\af_{[k+1,n]}}(C)]}s_{\af_{[j,n]}\af^m\af_{[1,n]}}^* \\
&=s_{\af_{[i,n]}\af^{l+1}}p_{[A\cap \theta_{\af_{[k+1,n]}}(C)]}s_{\af_{[j,n]}\af^{m+1}}^* \\
&=s_{\af_{[i,n]}\af^{l+1}}p_{[A]}s_{\af_{[j,n]}\af^{m+1}}^*. 
\end{align*}  
This means that the hereditary subalgebra  
$p_{[B]}C^*(\CB/(\CB\setminus\CT), \CL,\theta)p_{[B]}$ is generated by the elements 
$s_{\af_i}p_{[\theta_{\af_{[1,i]}}(A)]}$ for $1 \leq i \leq n$. 
Let $\gm$ be the restriction of the gauge action on 
$C^*(\CB/(\CB\setminus\CT), \CL,\theta)$ to the hereditary subalgebra 
$p_{[B]}C^*(\CB/(\CB\setminus\CT), \CL,\theta)p_{[B]}$ which is obviously gauge-invariant,  
and let $\bt$ be the gauge action of the universal 
(graph) $C^*$-algebra $C(\mathbb{T}) \otimes M_n$  
generated by 
the partial isometries $t_1,\dots,t_n$ 
satisfying the relations 
$$t_i^*t_i=t_{i+1}t_{i+1}^*,\, t_n^*t_n=t_1t_1^*,\, 
\text{ and }\, \sum_{j=1}^{n}t_j^*t_j=1$$
for $1 \leq i \leq n-1$. 
But the partial elements $s_{\af_i}p_{[\theta_{\af_{[1,i]}}(A)]}$, $1\leq i\leq n$, 
satisfy the above relations with $p_{[B]}$ in place of $1$, hence 
there exists a  homomorphism
  $$\pi:C(\mathbb{T}) \otimes M_n \rightarrow p_{[B]}C^*(\CB/(\CB\setminus\CT), \CL,\theta)p_{[B]}$$
such that $\pi(t_i)=s_{\af_i}p_{[\theta_{\af_{[1,i]}}(A)]}$ for $1 \leq i \leq n$. 
It is then immediate to have 
$\pi(\bt_z(t_i))=\gm_z(\pi(t_i))$ for all $i$ 
and thus the gauge-invariant uniqueness theorem (\cite[Theorem 5.10]{COP})
 proves that $\pi$ is an isomorphism. It then follows that  $C^*(\CB/(\CB\setminus\CT),\CL,\theta)$ contains an ideal that is not gauge-invariant. 
\end{proof}

We can now prove our main theorem.

\begin{thm}\label{equivalent:(K)}
Let $(\CB, \CL, \theta)$ be a Boolean dynamical system. Consider the  following. 
\begin{enumerate}
	\item[(1)]$(\CB, \CL, \theta)$ satisfies Condition (K).
	\item [(2)]$(\CB, \CL, \theta)$ has no cyclic maximal tails.
	\item[(3)] For every hereditary saturated ideal $\CH$ of $\CB$, the Boolean dynamical system $(\CB/\CH, \CL, \theta)$ satisfies Condition (L).
 \item[(4)] Every ideal in $C^*(\CB, \CL, \theta)$ is gauge-invariant.
  \end{enumerate}
We have  (1)$\iff$(2)$\iff$(3) and (4) implies each of  conditions (1)-(3). 
If moreover  $(\CB,\CL,\theta)$ is locally finite and $\CB$ and $\CL$ are countable, then all four conditions are equivalent. 
\end{thm}

\begin{proof}
 (1)$\implies$(2) follows from the definition of a cyclic maximal tail. (2)$\implies$(3) follows from Proposition~\ref{prop:K}, and (3)$\implies$(1) follows from Proposition~\ref{prop:cyclic maximal tails}.

\vskip 0.5pc
We show that  (4)$\implies$(2): Assume to the contrary that   $(\CB, \CL, \theta)$  has a cyclic maximal tail $\CT$. It then follows by Proposition \ref{prop2:cyclic maximal tails}
that  there is a $B\in\CT$ such that $p_{[B]}C^*(\CB/(\CB\setminus\CT),\CL,\theta)p_{[B]}$ is isomorphic to $M_n(C(\mathbb{T}))$ for some $n\in\mathbb{N}$.  
  Thus $C^*(\CB,\CL,\theta)$  has a quotient containing a corner that is isomorphic to $M_n(C(\mathbb{T}))$ for some $n\in\mathbb{N}$ and contains an ideal that is not gauge-invariant, a contradiction.

\vskip 0.5pc

Now assume that   $(\CB,\CL,\theta)$ is locally finite and $\CB$ and $\CL$ are countable and prove (3)$\implies$(4): 
Let $I$ be an ideal of $C^*(\CB, \CL,\theta)$. 
Then  $\CH_I=\{A \in \CB :p_A \in I\}$ is a 
hereditary saturated  ideal of $\CB $  and 
the ideal $I_{\CH_I}$ generated by the projections 
$\{p_A : A \in \CH_I\}$ is gauge-invariant.
Since $I_{\CH_I} \subseteq I$, the quotient map 
$$q: C^*(\CB, \CL,\theta)/I_{\CH_I} \rightarrow C^*(\CB, \CL,\theta)/I$$ 
given by $q(s+I_{\CH_I}):=s+I$ for $s\in C^*(\CB, \CL,\theta)$, 
is well-defined.
From \cite[Proposition 10.11]{COP}, we have an isomorphism 
$\pi: C^*(\CB /\CH_I , \CL,\theta)\to C^*(\CB, \CL,\theta)/I_{\CH_I}$
which maps the canonical generators to  the canonical generators.   
Then the composition map  
$q \circ \pi :C^*(\CB /\CH_I , \CL,\theta) \rightarrow  C^*(\CB, \CL,\theta )/I$ 
satisfies  
\begin{align*}
q \circ \pi(p_{[A]})& =q(p_A+I_{\CH_I})=p_A+I\\
q \circ \pi(s_\af) & =q(s_\af+I_{\CH_I})=s_\af+I 
\end{align*} 
for $[A]\in \CB /\CH_I$ and $\af \in \CL$. 
If $p_{[A]}\neq 0$, then  $[A]\neq [\emptyset]$ in $\CB /\CH_I$, hence 
 $A \notin \CH_I$.
Thus  $p_A+I\in C^*(\CB, \CL,\theta)/I$ is a nonzero projection.
Since the quotient Boolean dynamical system $(\CB /\CH_I , \CL,\theta)$ satisfies Condition $(L)$,
we see that the map $q \circ \pi $ is injective by the Cuntz-Krieger Uniqueness Theorem~\ref{CK uniqueness thm}.  
   Thus  $q$ is injective, so that $I$ must coincide with 
   the gauge-invariant ideal  $I_{\CH_I}$. 
       
   Note that if $I$ is an ideal such that $\CH_I=\{ \emptyset \}$, 
   then  $I_{\CH_I}=\{0\}$, and   
   $q \circ \pi$ is the quotient map $q: C^*(\CB, \CL,\theta) \rightarrow C^*(\CB, \CL,\theta)/I$. 
   Then  the family $\{p_A+I,\, s_\af+I \}$ is a 
   representation of the Boolean dynamical system $(\CB, \CL,\theta)$ 
   in the $C^*$-algebra $C^*(\CB, \CL,\theta)/I$ such that 
   $p_A+ I \neq 0$ and $s_\af+I \neq 0$ for each $A \in \CB $ and $\af \in \CL$.
Since $(\CB,\CL,\theta)$ satisfies Condition $(L)$, the Cuntz-Krieger Uniqueness Theorem \ref{CK uniqueness thm} again says that $q$ is injective.
  Thus we have $I=\{0\}$. 
\end{proof}

\section{The primitive ideal space of $C^*(\CB,\CL,\theta)$}\label{primitive ideal space}
In this section, we describe the primitive ideal space of  the $C^*$-algebra of a locally finite Boolean dynamical system that satisfies Condition (K) and has countable $\CB$ and $\CL$. The notion of a maximal tail   plays a crucial role in characterizing primitive gauge-invariant ideals of $C^*(\CB,\CL,\theta)$.  We   analyze the topology on the set all maximal tails $\mathrm{M}$ in $\CB$ and  then show  that 
 $\mathrm{M}$ is homeomorphic to the primitive ideal space of $C^*(\CB,\CL, \theta)$.

\subsection{The space of maximal tails}

Let $(\CB,\CL,\theta)$ be a Boolean dynamical system. Denote by $\mathrm{M}$ the set of all maximal tails in $\CB$. 
\vskip 1pc

\begin{prop}\label{max-top} Let $(\CB,\CL,\theta)$ be a Boolean dynamical system. For $A \in \CB$, let $$\mathrm{U}_A:=\{\CT \in \mathrm{M}: A \in \CT\}.$$
Then $\{\mathrm{U}_A:A \in \CB \}$ is a basis of compact open sets for a topology on $\mathrm{M}$.
\end{prop}
\vskip 1pc

\begin{proof} It is obvious that $\cup_{A \in \CB} ~\mathrm{U}_A = \mathrm{M}$.
For $A_1, A_2 \in \CB$, choose $\CT \in \mathrm{U}_{A_1} \cap \mathrm{U}_{A_2}$. Then $A_1, A_2 \in \CT$, and thus there exists $C \in \CT$ such that $A_1 \geq C$ and $A_2 \geq C$ by (T5). So there are $\af_i \in  \CL^*$ so that $C \subseteq \theta_{\af_i}(A_i)$ for $i=1,2$. Hence by (T3), we see that $\theta_{\af_i}(A) \in \CT$ for $i=1,2$. Therefore, $A_1, A_2 \in \CT$ by (T1).  Thus it follows that  $$ \CT \in\mathrm{U}_C \subseteq \mathrm{U}_{A_1} \cap \mathrm{U}_{A_2}.$$ Thus, $\{\mathrm{U}_A:A \in \CB \}$ forms a basis of a topology of $\mathrm{M}$.  To show $\mathrm{U}_A$ is compact, 
 consider the injective map $\iota: \mathrm{M} \rightarrow \{0,1\}^{\CB}$ given by 
$$\iota(\CT)(A)=\left\{
                      \begin{array}{ll}
                       1 & \hbox{if\ }~~ A \in \CT \\
                      0 & \hbox{if\ }~~ A \notin \CT. \\
                                                                 \end{array}
                    \right.
$$ for  $A \in \CB$. Equip $\{0,1\}^{\CB}$ with the product topology. Then $\{0,1\}^{\CB}$ is compact. Since $$\iota(\mathrm{U}_A) 
=\{\eta \in \{0,1\}^{\CB} : \eta(A)=1\} = \pi_A^{-1}(\{1\}),
$$
where $\pi_A: \{0,1\}^{\CB} \rightarrow \{0,1\}$ is defined by $\pi_A(f)=f(A)$, 
the map $\iota: \mathrm{M} \rightarrow \{0,1\}^{\CB}$ is an open map.  
 Thus, $\iota: \mathrm{M} \rightarrow \iota(\mathrm{M})$ is a bijective open map, and hence  $\iota^{-1}: \iota(\mathrm{M}) \rightarrow \mathrm{M}$ is continuous. So, 
 $\mathrm{U}_A=\iota^{-1}(\iota(\mathrm{U}_A))$ is compact since $\iota(\mathrm{U}_A) $ is compact in $\{0,1\}^{\CB}$. 
\end{proof}

We shall now characterize the closed subsets of $\mathrm{M}$. For a subset $\mathrm{S} $ of $\mathrm{M}$, denote by $\overline{\mathrm{S}}$ the closure of $\mathrm{S}$ in $\mathrm{M}$.

\begin{lem} \label{closure}
Let $(\CB,\CL,\theta)$ be a  Boolean dynamical system and $\mathrm{S}$ a subset of $\mathrm{M}$. Then $\overline{\mathrm{S}}=\{\CT \in \mathrm{M}: \CT \subseteq \cup_{\CS \in \mathrm{S}}\CS\}$.
\end{lem}

\begin{proof}
We have
\begin{align*} 
\CT \in \overline{\mathrm{S}} & \iff \mathrm{U}_A \cap \mathrm{S} \neq \emptyset ~\text{for all}~ A \in \CT \\
& \iff \CT \subseteq \cup_{\CS \in \mathrm{S}} \CS.
\end{align*}
\end{proof}

\subsection{The primitive ideal spaces}

If a locally finite Boolean dynamical system $(\CB,\CL,\theta)$ with countable $\CB$ and $\CL$ satisfies Condition (K), then every ideal has the form $I_\CH$ for some hereditary saturated ideal $\CH$ of $\CB$ by Theorem \ref{equivalent:(K)}.
So, we only need to  determine when the gauge-invariant ideal $I_\CH$ is primitive.
We start with the following  lemma that holds true without assuming locally finiteness of $(\CB,\CL,\theta)$ and countability of $\CB$ and $\CL$. 

\begin{lem}\label{prim gives max} Let $(\CB,\CL,\theta)$ be a  Boolean dynamical system. If $I $ is a primitive ideal of $C^*(\CB,\CL, \theta)$, then $\CT:=\{A \in \CB : p_A \notin I \}$ is a maximal tail of $\CB$. 
\end{lem}

\begin{proof} The set $\CH:=\{A \in \CB: p_A \in I\}$ is a proper hereditary saturated ideal of $\CB$  (see for example, \cite[Lemma 3.5]{JKP}). So, $\CT=\CB \setminus\CH$ satisfies \ref{T0}, \ref{T1}, \ref{T2}, \ref{T3}, and \ref{T4}. To show \ref{T5}, choose $A_1, A_2 \in \CT$ and take an irreducible representation $\pi: C^*(\CB,\CL,\theta) \rightarrow \mathfrak{B}(H_{\pi})$ such that $\ker(\pi)=I$. Since $$\CT=\{A \in \CB : p_A \notin I\},$$ we have $p_{A_1} \notin I$, and hence $\pi(p_{A_1})H_{\pi} \neq \{0\}$. Similarly, the space $\pi(p_{A_2})H_{\pi}$ is also non-trivial subspace of $H_{\pi}$. Fix $h \in \pi(p_{A_1})H_{\pi}$ so that $\|h\|=1$. Since $\pi$ is irreducible, $h$ is cyclic for $\pi$, so that there exists $a \in C^*(\CB,\CL,\theta)$ such that $\pi(p_{A_2})\pi(a)h=\pi(p_{A_2}ap_{A_1})h \neq 0$. In particular, we have $\pi(p_{A_2}ap_{A_1}) \neq 0$. Since 
$$p_{A_1}(s_{\mu}p_Bs_{\nu}^*)p_{A_2}=s_{\mu}p_{\theta_{\mu}(A_1) \cap B \cap \theta_{\nu}(A_2)}s_{\nu}^* \neq 0$$ 
only if $\theta_{\mu}(A_1) \cap B \cap \theta_{\nu}(A_2) \neq \emptyset$, we see that  
$$\pi(p_{A_2}ap_{A_1}) \in \overline{\operatorname{span}}\{\pi(s_{\mu}p_  Cs_{\nu}^*): \mu, \nu \in \CL^*,  C \in \CB , \emptyset \neq C \subseteq \theta_{\mu}(A_1) \cap \theta_{\nu}(A_2) \} \setminus \{0\}.$$
Thus there exist $\mu, \nu \in \CL^*$ and $C \subseteq \theta_{\mu}(A_1) \cap \theta_{\nu}(A_2)$ such that $\pi(s_{\mu}p_  Cs_{\nu}^*) \neq 0$. One can also shows that $\pi(p_C) \neq 0$, giving $p_C \notin I$. So $C \in \CT$ satisfies $A_1 \geq C$ and $A_2 \geq C$. Therefore, $\CT$ is a maximal tail. 
\end{proof}
\vskip 0.5pc

\begin{prop}\label{max-bij-prim} Let $(\CB,\CL,\theta)$ be a locally finite Boolean dynamical system which satisfies Condition $(K)$ and has countable $\CB$ and $\CL$.  Then  $\CH$ is a hereditary saturated ideal of $\CB$ such that 
 $I_\CH$ is primitive
if and only if $\CT:= \CB \setminus \CH$ is a maximal tail.
\end{prop}
\begin{proof} ($\Rightarrow$) Suppose that   $\CH$ is a hereditary saturated ideal of $\CB$ and $I_\CH$ is primitive. Then by Lemma \ref{prim gives max}, the set $\CT:=\CB \setminus \CH=\{A \in \CB : p_A \notin I_\CH\}$ is a maximal tail. 

$(\Leftarrow)$ Let $\CT$ be a maximal tail. Then $\CH:=\CB \setminus \CT$ is a proper hereditary and saturated ideal in $\CB$ by Remark \ref{remark:1}.  We show that $I_\CH$ is a prime ideal.  Suppose that $I_1, I_2$ are ideals in $C^*(\CB,\CL,\theta)$ such that $I_1 \cap I_2 \subseteq I_\CH$. Since every ideal of $C^*(\CB,\CL,\theta)$ is gauge-invariant by Theorem \ref{equivalent:(K)}, it follows from \cite[Proposition 10.11]{COP} that there are hereditary saturated subsets $\CH_i$ such that $I_i=I_{\CH_i}$ for $i=1,2$. Then $\CH_1 \cap \CH_2 \subseteq \CH$. If $\CH_1 \nsubseteq \CH$ and $\CH_2 \nsubseteq \CH$, then there are $A_i \in \CH_i \setminus \CH$ for $i=1,2$. By \ref{T5}, there exists $C \in \CT$ such that $A_1 \geq C$ and $A_2 \geq C$. So there are $\af_i \in \CL^*$ such that $C \subseteq \theta_{\af_i}(A_i)$ for $i=1,2$. Then $C \in \CH_1 \cap \CH_2 \subseteq \CH$ since $\CH_1$ and $\CH_2$ are hereditary. This contradicts  $C \notin \CH$. Hence, either $\CH_1 \subseteq \CH$ or $\CH_2 \subseteq \CH$, which means either $I_1=I_{\CH_1} \subseteq I_\CH$ or $I_2=I_{\CH_2} \subseteq I_\CH$. This shows that $I_\CH$ is prime. Thus, it is primitive since $C^*(\CB,\CL,\theta)$ is separable (\cite[Proposition A.17 and Remark A.18]{RW}).  
\end{proof}

 By $\prim(C^*(\CB, \CL, \theta))$ 
we mean the set of primitive ideals in $C^*(\CB, \CL, \theta)$. 
We now  obtain a complete list of primitive ideals  of $C^*(\CB,\CL,\theta)$ and a description of 
 the hull-kernel topology of $\prim(C^*(\CB,\CL,\theta))$.

\begin{thm}\label{max-homeo-prim}Let $(\CB,\CL,\theta)$ be a locally finite Boolean dynamical system such that $\CB$ and $\CL$ are countable. Suppose that $(\CB,\CL,\theta)$ satisfies Condition (K).   Then the map $$\phi: \mathrm{M} \rightarrow \prim(C^*(\CB,\CL,\theta))$$ defined by $\phi(\CT)=I_{H_{\CT}}$ is a homeomorphism, where $H_{\CT}=\CB \setminus \CT$.
\end{thm}

\begin{proof}

By Proposition \ref{max-bij-prim}, we see that $\CT\mapsto I_{\CB\setminus\CT}$ is a surjective map from $\mathrm{M}$ to $\prim(C^*(\CB, \CL, \theta))$. It follows from Lemma \ref{prim gives max}  that $\CT=\{A\in\CB:p_A\notin I_{\CB\setminus\CT}\}$, so $\phi$ is injective. If $\mathrm{S}$ a subset of $\mathrm{M}$, then it follows from Lemma~\ref{closure} that 
$$\phi(\overline{\mathrm{S}})=\{I_{\CB\setminus\CT}:\CT\subseteq\cup_{\mathcal{S}\in\mathrm{S}}\mathrm{S}\}=\{I\in \prim(C^*(\CB, \CL, \theta)): \cap_{\mathcal{S}\in\mathrm{S}} I_{\CB\setminus\mathcal{S}} \subseteq I\} =\overline{\phi(\mathrm{S})}.$$ 
This shows that $\phi$ is a homeomorphism.
\end{proof}

\section{Topological dimension zero}\label{top dim zero}

A $C^*$-algebra $A$ is said to have {\em topological dimension zero} if the primitive ideal space of $A$ endowed with the hull-kernel topology has a basis of compact open sets (\cite{BP2}).
Proposition \ref{max-bij-prim} and Theorem \ref{max-homeo-prim} say that if  a locally finite Boolean dynamical system $(\CB,\CL,\theta)$ satisfies Condition (K) and has countable $\CB$ and $\CL$, then the topological dimension of $C^*(\CB, \CL, \theta)$ is 0. 
We  show in Theorem \ref{equivalent:(K):top dim zero}  that the converse is also true. 
We also show that this is equivalent to $C^*(\CB,\CL,\theta)$ having the (weak) ideal property.

To begin with, we recall that 
a $C^*$-algebra $A$ is said to  have {\em the ideal property}  (\cite[Remark 2.1]{CPR}) if whenever $I,J$ are ideals in $A$ such that $I$ is not contained in $J$, there is a projection in $I \setminus J$. A  $C^*$-algebra $A$ is said to have {\em the weak ideal property} (\cite[Definition 8.1]{PP}) if whenever $I \subsetneq J \subset \mathcal{K} \otimes A $ are ideals in $ \mathcal{K} \otimes A$,  where $\mathcal{K}$ denotes the $C^*$-algebra of compact operators on a separable infinite dimensional Hilbert space, then $J/I$ contains a nonzero projection.

\begin{thm}\label{equivalent:(K):top dim zero}
Let $(\CB,\CL,\theta)$ be a locally finite Boolean dynamical system such that $\CB$ and $\CL$ are countable. Then the following are equivalent.
\begin{enumerate}
  \item $(\CB, \CL, \theta)$ satisfies Condition $(K)$.
 \item $C^*(\CB, \CL, \theta)$ has the ideal property.
  \item $C^*(\CB, \CL, \theta)$ has the weak ideal property.
     \item The topological dimension of $C^*(\CB, \CL, \theta)$ is 0.
  \item $C^*(\CB, \CL, \theta)$ has no quotients containing a corner that is isomorphic to $M_n(C(\mathbb{T}))$.
 \end{enumerate}
\end{thm}

\begin{proof} 
(1)$\implies$(2):  Let $I$ and $J$ be ideals of $C^*(\CB, \CL,\theta)$ such that $I\not\subseteq J$. Then $I$ and $J$ are gauge-invariant by Theorem \ref{equivalent:(K)}. Therefore it follows from \cite[Theorem 10.12]{COP} that $\mathcal{H}_I:=\{A\in\CB:p_A\in I\}\not\subseteq \{A\in\CB:p_A\in J\}=:\mathcal{H}_J$. Thus, $J/I$ contains a projection.

(2)$\implies$(3): Follows from \cite[Proposition 8.2]{PP}.

(3)$\implies$(4): Follows from \cite[Theorem 1.8]{PP2}.

(4)$\implies$(5): It follows from \cite[Proposition 3.2.1]{Dixmier} and \cite[Proposition 2.6]{BP2} that the property of having topological dimension zero passes to ideals and  quotients. Since Morita equivalent $C^*$-algebras have homeomorphic primitive ideal spaces (see for instance \cite[p.156]{Connes}), topological dimension zero passes to full corners, and thus to corners. Since $M_n(C(\mathbb{T}))$ does not have topological dimension zero, it follows that (4)$\implies$(5).

(5)$\implies$(1):  Suppose that   $(\CB, \CL, \theta)$ does not satisfy Condition (K). Then once again  $(\CB, \CL, \theta)$  has a cyclic maximal tail $\CT$ and  there is a $B\in\CT$ such that $p_{[B]}C^*(\CB/(\CB\setminus\CT),\CL,\theta)p_{[B]}$ is isomorphic to $M_n(C(\mathbb{T}))$ for some $n\in\mathbb{N}$ by Proposition \ref{prop2:cyclic maximal tails}.  
  Thus $C^*(\CB,\CL,\theta)$  has a quotient containing a corner that is isomorphic to $M_n(C(\mathbb{T}))$ for some $n\in\mathbb{N}$, a contradiction. 

\end{proof}

\begin{cor} \label{cor} Let $(\CB,\CL,\theta)$ be a locally finite Boolean dynamical system such that $\CB$ and $\CL$ are countable. If $C^*(\CB,\CL,\theta)$ has real rank zero or is purely infinite, then $(\CB, \CL, \theta)$ satisfies Condition $(K)$.
\end{cor}

\begin{proof}
Suppose first that $C^*(\CB,\CL,\theta)$ has real rank zero. It then follows from \cite[Theorem 2.6]{BP} that $C^*(\CB, \CL, \theta)$ has the ideal property, and thus from Theorem~\ref{equivalent:(K):top dim zero} that $(\CB, \CL, \theta)$ satisfies Condition $(K)$.

Suppose then that $C^*(\CB,\CL,\theta)$ is purely infinite and that $A$ is a corner of a quotient of $C^*(\CB,\CL,\theta)$. Since the property of being purely infinite passes to quotients and corners \cite[Propositions 4.3 and 4.17]{KR}, $A$ is purely infinite and thus cannot be isomorphic to $M_n(C(\mathbb{T}))$. It therefore follows from Theorem~\ref{equivalent:(K):top dim zero} that $(\CB, \CL, \theta)$ satisfies Condition $(K)$.
\end{proof}

\vskip 3pc
\subsection*{Acknowledgements}
The discussion of this paper started when the second author visited the first author in Faroe Islands.
The second author is grateful for the hospitality of the first author and his family. 

The authors would like to thank the referee for careful reading and valuable comments.


\end{document}